\documentclass[10 pt,reqno]{amsart}
\usepackage{amsmath}
\usepackage{amsfonts}
\usepackage{amssymb}
\usepackage{amsthm}
\usepackage{amscd}
\usepackage{a4wide}
\usepackage{enumerate}
\usepackage{dsfont} 
\usepackage{xcolor}
\usepackage{graphics}
\usepackage{eucal}
\usepackage{mathrsfs}
\usepackage[original]{imakeidx}
\usepackage{mathtools}
\usepackage{mdwlist}
\usepackage{environ}
\usepackage{enumitem}
\usepackage{multicol}
\setenumerate[1]{label=\thesection.\arabic*.}
\newif\ifhinting\hintingtrue 
\mathtoolsset{showonlyrefs}

\usepackage{color}
\usepackage[pdftex,bookmarks,colorlinks,breaklinks]{hyperref}  

\hypersetup{linkcolor=red,citecolor=blue,filecolor=dullmagenta,urlcolor=darkblue} 

\newtheorem {theorem}    {Theorem}[section]
\newtheorem {lemma}      [theorem]    {Lemma}

\newtheorem {proposition}[theorem]    {Proposition}

\theoremstyle{definition}
\newtheorem {definition} [theorem]    {Definition}

\newtheorem {remark}    [theorem]    {Remark}

\newcounter{AbcT}

\numberwithin{equation}{section}

\newcommand{\IGNORE}[1]{}

\DeclareMathOperator{\supp}{supp}

 \DeclareMathOperator{\SL}{SL}

\DeclareMathOperator{\GL}{GL} 
 \DeclareMathOperator{\SO}{SO}

\DeclareMathOperator{\Ad}{Ad}

\newcommand\vol{\operatorname{vol}}

\newcommand{\diag}{\operatorname{diag}}
\newcommand\Lie{\operatorname{Lie}}

\begin{document}
\title{  Central limit theorems for lattice point counting on tessellated domains}

\begin{abstract}
Following the approach of Bj$\ddot{\text{o}}$rklund and Gorodnik, we have considered the discrepancy function for lattice point counting on domains that can be nicely tessellated by the action of a diagonal semigroup. We have shown that suitably normalized discrepancy functions for lattice point counting on certain tessellated domains satisfy a non-degenerate central limit theorem. Furthermore, we have also addressed the same problem for affine and congruence lattice point counting, proving analogous non-degenerate central limit theorems for them. The main ingredients of the proofs are the method of cumulants and quantitative multiple mixing estimates.


\end{abstract}
\author{Sourav Das}
 
\address{ Department of Mathematics, Ben-Gurion University of the Negev, Beer Sheva 8410501, Israel
}
\email{iamsouravdas1@gmail.com, dassou@post.bgu.ac.il}
\maketitle
\section{Introduction}\label{section:intro}
One of the central problems in the geometry of numbers is the lattice point counting problem in certain domains of Euclidean space. Minkowski's first theorem is one of the fundamental results in the geometry of numbers, which says that a centrally symmetric convex set $C$ of $\mathbb R^k$ with $\vol(C)>2^k$ contains a non-zero integer point. This theorem can be extended to a general lattice $g\mathbb Z^k$ for some $g \in \GL_k(\mathbb R).$ Before proceeding further, let us set a few notations. The space of all unimodular lattices of $\mathbb R^k$ is denoted by $X_k,$ which is naturally identified with $\SL_k(\mathbb R)/ \SL_k(\mathbb Z),$ which supports a natural $\SL_k(\mathbb R)$-invariant measure $\mu_k$ coming from the Haar measure on $\SL_k(\mathbb R)$ such that $\mu_k(X_k)=1.$ Throughout this article, given a finite set $B$, we denote the number of elements of $B$ by $|B|$. In $1945, $ Siegel (\cite{Si}) proved a probabilistic result of the first of its kind in the geometry of numbers. Suppose that $\Omega$ is a measurable subset of $\mathbb R^k,$ not containing zero. Then Siegel's mean value theorem (\cite{Si}) states that if we choose an unimodular lattice $\Lambda$ at random, on average $\Lambda $ contains $\vol(\Omega)$ many points of $\Omega.$ Along the same lines, Rogers (\cite{R1}) gives formulas for higher moments for the counting function $|\Lambda \cap \Omega|.$ To date, Siegel’s mean value theorem and Rogers's formulas have applications in numerous number theoretic and ergodic problems, such as lattice point counting, Oppenheim-type problems related to quadratic forms, effective ergodic theorems, and more.

There has been extensive study to understand the behavior of the error term $|\Lambda \cap \Omega| - \vol(\Omega)$ of lattice point counting problems for a varying, sufficiently nice family of domains. Let $\{\Omega_T\}$ be an increasing family of finite volume Borel subsets of $\mathbb R^k$ such that $\vol(\Omega_T) \rightarrow \infty$ as $ T \rightarrow \infty .$ In \cite{S1}, Schmidt prove that for $\mu_k$ almost every unimodular lattice $\Lambda$,
\begin{equation}\label{equ:11}
|\Lambda \cap \Omega_T|=\vol(\Omega_T) + O\left(\vol(\Omega_T)^{1/2} \log^2(\vol(\Omega_T))\right).
\end{equation}
Note that Schmidt's result was much more general than this; for simplicity, we are considering it in this form. The key ingredient in proving \eqref{equ:11} is Rogers's second moment formula (in particular, the discrepancy bound). One must note that we can view Schmidt's result as an analog of the Law of Large Numbers. This motivates the question of whether other probabilistic limit laws, such as the central limit theorem, the law of iterated logarithm, etc., also hold for the above kind of lattice point counting problems. In the paper \cite{BG2}, Bj$\ddot{\text{o}}$rklund  and Gorodnik studied the discrepancy functions given by
\begin{equation}\label{12}
    D_T(\Lambda)=|\Lambda \cap \Omega_T| - \vol(\Omega_T).
\end{equation}
and partially answered the above question by proving that suitably normalized discrepancy functions satisfy a non-degenerate central limit theorem for domains defined by products of linear forms. Our current project is devoted to proving certain central limit theorems for discrepancy functions of unimodular lattices, affine unimodular lattices, and congruence unimodular lattices for domains, which can be nicely tessellated by means of the action of a diagonal subgroup. The domains we have considered appear naturally in lattice point counting problems and Diophantine approximations (see \cite{DFV,BG1,AG}). It is worth mentioning a few recent works in this context. Recently, K. Holm proved a central limit theorem for symplectic lattice point counting in \cite{H}. Using the method of cumulants and quantitative multiple mixing estimate, Bj$\ddot{\text{o}}$rklund and Gorodnik (\cite{BG1}) proved a central limit theorem for Diophantine approximations. Prior to this, the results of this paper for unweighted cases were proved by Dolgopyat, Fayad, and Vinogradov in \cite{DFV}, although their approach was completely different from that of \cite{BG1}. In a very recent paper \cite{AG}, Aggarwal and Ghosh extended the results of \cite{BG1} for inhomogeneous Diophantine approximation and Diophantine approximation with congruence conditions.

\subsection{The space of affine unimodular lattices}\label{subsec:affine} We denote the space of affine unimodular lattices of $\mathbb R^k$ by $X_{k,a}.$ $\SL_k(\mathbb R) \ltimes \mathbb R^k $ acts naturally on $X_{k,a}$. The action is given by
$$ ((g,\mathbf{v}),\Lambda) \mapsto g\Lambda +\mathbf{v}, \,\, \,\, \text{for} \,\, (g,\mathbf{v}) \in \SL_k(\mathbb R) \ltimes \mathbb R^k \,\, \text{and} \,\, \Lambda \in X_{k,a}.$$
Note that the action is transitive and the stabilizer of $\mathbb Z^k$ is $ \SL_k(\mathbb Z) \ltimes \mathbb Z^k$. Hence, we have the following identification
 $$X_{k,a} \cong \SL_k(\mathbb R) \ltimes \mathbb R^k /  \SL_k(\mathbb Z) \ltimes \mathbb Z^k.$$

\noindent Since $\SL_k(\mathbb Z) \ltimes \mathbb Z^k$ is a lattice in $\SL_k(\mathbb R) \ltimes \mathbb R^k$, $X_{k,a}$ supports a natural $\SL_k(\mathbb R) \ltimes \mathbb R^k $-invariant probability measure $\mu_{k,a}$, coming from the left-invariant Haar measure on $\SL_k(\mathbb R) \ltimes \mathbb R^k. $ 
 
Let $G_k=\SL_k(\mathbb R)$ and $G_{k,a}= \SL_k(\mathbb R) \ltimes \mathbb R^k$. Then $G_{k,a}$ can be identified as a subgroup of $G_{k+1}$ in view of the map $h:G_{k,a} \rightarrow G_{k+1}$ defined by
\begin{equation}\label{equ:map_h}
    (g,\mathbf{v}) \mapsto \begin{bmatrix}
   g& \mathbf{v}\\0& 1
\end{bmatrix}
\quad \text{for} \,\, (g,\mathbf{v}) \in G_{k,a}.
\end{equation}
Clearly $h$ induces an injective map $\Tilde{h}:X_{k,a} \rightarrow X_{k+1}.$ Furthermore there is surjection $\pi : G_{k,a} \rightarrow G_k $ given by 
\begin{equation}
    \pi(g,\mathbf{v})=g \quad \text{for all} \,\, (g,\mathbf{v}) \in G_{k,a}.
\end{equation}
$\pi$ also induces a surjection $\Tilde{\pi}:X_{k,a} \rightarrow X_k.$ It is easy to observe that $\Tilde{\pi}$ is measure preserving.  
\subsection{The space of congruence lattices}
 Let $(\mathbf v,N ) \in \mathbb Z^k \times \mathbb N$ be such that $\gcd(\mathbf v, N)=1.$ Denote the space of all affine unimodular lattices of $\mathbb R^k$ of the form $g(\mathbb Z^k + \frac{\mathbf v}{N})$ by $X_{k,c}.$ We can identify $X_{k,c}$ with the homogeneous space $\SL_k(\mathbb R)/\Gamma_{\mathbf v,N}$ via 
 $$ g(\mathbb Z^k + \frac{\mathbf v}{N}) \mapsto g \Gamma_{\mathbf v,N},$$
 where $\Gamma_{\mathbf v,N}=\{g \in \SL_k(\mathbb Z): g \mathbf v \equiv \mathbf v \,\, (\text{mod N})\}$ is the stabilizer of $(\mathbb Z^k +\frac{\mathbf v}{N})$ in $\SL_k(\mathbb Z)$. Note that $\Gamma_{\mathbf{v},N}$ is a congruence subgroup, since it contains the principal congruence subgroup $\Gamma_N=\{g\in \SL_k(\mathbb Z):g \equiv Id \,\,(\text{mod N})\}.$ $\Gamma_N$ is a finite index subgroup (indeed it is normal) of $\SL_k(\mathbb Z)$ and $\SL_k(\mathbb Z)/\Gamma_N \cong \SL_k(\mathbb Z/N\mathbb Z).$ Thus $\Gamma_{\mathbf{v},N}$ is also a finite index subgroup of $\SL_k(\mathbb Z),$ and hence also a lattice in $\SL_k(\mathbb R).$ Since $\Gamma_{\mathbf v,N}$ is a lattice in $\SL_k(\mathbb R),$ there exists a $\SL_k(\mathbb R)$-invariant measure $\mu_{k,c}$ on $X_{k,c}$ coming from the Haar measure on $\SL_k(\mathbb R)$ such that $\mu_{k,c}(X_{k,c})=1.$

 {It is worth mentioning that congruence lattices appear very naturally in many places, for example, in the context of Diophantine approximation with congruence conditions.}

\subsection{Main results}
In this article, we are interested in studying the discrepancy functions on $X_k$, $X_{k, a}$, and $X_{k,c}$ for the domains of the following type.\\ 
Let $c_1,\dots,c_m \in \mathbb R_{>0}$    and $u_1,\dots,u_m\in \mathbb R_{>0}$ be such that $u_1 +\dots +u_m=n.$ Now for $T>0$, we consider the domains of the form
\begin{equation}\label{equ:domain}
    \Omega_T =\left\{(\mathbf x,\mathbf y) \in \mathbb R^{m} \times \mathbb R^n:   1 \leq \|\mathbf y\| \leq T, \,\,  |x_i|\|\mathbf y\|^{u_i} < c_i \,\, \text{for}\,\, i=1,\dots,m \right \},
\end{equation}
where $\mathbf x =(x_1,\dots,x_m),$ $\|\cdot\|$ be the usual Euclidean norm on $\mathbb R^n$ and $|\cdot|$ is the usual absolute value on $\mathbb R.$ Let
\begin{equation}\label{equ:normal_dist}
N_{\xi}(0,\sigma^2):=\frac{1}{\sqrt{2\pi \sigma}} \int_{-\infty}^{\xi} e^{\frac{-x^2}{2\sigma}} \,\, dx,
\end{equation}
denotes the normal distribution function with mean $0$ and variance  $\sigma^2$, $\zeta(x)=\sum_{j=1}^{\infty} \frac{1}{j^x}$ denotes the Riemann zeta function and $\omega_n$ denotes the $(n-1)$-dimensional surface area of the unit sphere $\mathbb S^{n-1}=\{\mathbf{x} \in \mathbb R^n:\|\mathbf{x}\|=1\}.$\\

 We now present the main results of this article, which state that the normalized discrepancy functions associated with the respective lattice (unimodular, affine unimodular, or congruence lattices) point counting problem over $\Omega_T$ satisfy non-degenerate central limit theorems. The following result is for unimodular lattices:
\begin{theorem}\label{thm:main_lattice}
    Let $\Omega_T$ be as in \eqref{equ:domain} and {$m+n \geq 4$.} Then  for any $\xi \in \mathbb R$
    \begin{equation}
        \mu_{m+n}\left\{ \Lambda \in X_{m+n}:\frac{|\Lambda \cap \Omega_T| - \vol(\Omega_T)}{\vol(\Omega_T)^{1/2}} < \xi \right \} \longrightarrow N_{\xi}(0,\sigma_u^2) \,\,\,\, \text{as} \,\, T \rightarrow \infty,
    \end{equation}
    where 
    \begin{equation}\label{equ:vari_unimod}
    \sigma_u^2 := 2 \left(\frac{2\zeta(m+n-1)}{\zeta(m+n)}-1\right).
    \end{equation}
     
\end{theorem}
\noindent We also have the analogous result for affine unimodular lattices:
\begin{theorem}\label{thm:main_affine}
    Let $\Omega_T$ be as in \eqref{equ:domain} and {$m+n \geq 4$.} Then  for any $\xi \in \mathbb R$
    \begin{equation}\label{equ:affine_main_conv}
        \mu_{m+n,a}\left\{ \Lambda \in X_{m+n,a}:\frac{|\Lambda \cap \Omega_T| - \vol(\Omega_T)}{\vol(\Omega_T)^{1/2}} < \xi \right \} \longrightarrow N_{\xi}(0,1) \,\,\,\, \text{as} \,\, T \rightarrow \infty.
    \end{equation}
     
\end{theorem}
\noindent Furthermore, we have an analogous central limit theorem for congruence lattices:
 \begin{theorem}\label{thm:main_congruence}
     Let $\Omega_T$ be as in \eqref{equ:domain} and {$m+n \geq 4$.} Then for any $\xi \in \mathbb R$
    \begin{equation}
        \mu_{m+n,c}\left\{ \Lambda \in X_{m+n,c}:\frac{|\Lambda \cap \Omega_T| - \vol(\Omega_T)}{\vol(\Omega_T)^{1/2}} < \xi \right \} \longrightarrow N_{\xi}(0,\sigma_c^2) \,\,\,\, \text{as} \,\, T \rightarrow \infty,
    \end{equation}
    where 
    \begin{equation}
    \sigma_c^2 :=  \frac{ 1}{\zeta_N(m+n)}  \sum_{\tiny \begin{array}{c} s_1 \geq 1\\ \gcd(s_1,N)=1 \end{array} } \sum_{ \tiny \begin{array}{c} s_2 \in \mathbb Z \setminus \{0\}  \\ s_2 \equiv s_1  (\mod  N) \end{array}} \max \{s_1,|s_2|\}^{-(m+n)},
    \end{equation}
     and $$\zeta_N(x)=\sum_{\tiny \begin{array}{c} s \geq 1\\ \gcd(s,N)=1 \end{array} } s^{-x} .$$ 
\end{theorem}

\begin{remark}\label{remark:spi}$\ $
    \begin{enumerate}[label=({{\arabic*}})]
    \item The domains of the form $\Omega_T$ naturally appear in problems of Diophantine approximation (see \cite{DFV,BG1,AG}). One should note that the domains $\Omega_T$ are different from the domains considered in \cite{BG2}. Domains of type $\Omega_T$ have been studied before in \cite{DFV,BG1,AG}, where all the previous authors proved central limit theorems for Diophantine approximates. Also, it is worth mentioning that our main results do not seem to follow immediately from the main results of \cite{DFV,BG1,BG2,AG}.
    \item {Note that Theorem \ref{thm:main_lattice} is a particular case of Theorem \ref{thm:main_congruence}. In fact, for $v \in \mathbb Z^{m+n}$ and $N=1$, $X_{m+n,c}=X_{m+n}.$ In the case of $N=1$, i.e., unimodular lattices, the expression for the variance can be simplified as follows:}
    \begin{eqnarray}
         \sigma_c^2 &=&  \frac{ 1}{\zeta(m+n)}  \sum_{\tiny \begin{array}{c} s_1 \geq 1\\   \end{array} } \sum_{ \tiny \begin{array}{c} s_2 \in \mathbb Z \setminus \{0\}  \end{array}} \max \{s_1,|s_2|\}^{-(m+n)}  \nonumber \\ &=& \frac{2}{\zeta(m+n)} \sum_{s_1,s_2\geq 1} \max \{s_1,s_2\}^{-(m+n)} \nonumber \\ &=& \frac{2}{\zeta(m+n)}\sum_{s=1}^{\infty} (2s-1)s^{-(m+n)} \nonumber \\ &=& 2\left(\frac{2\zeta(m+n-1)}{\zeta(m+n)}-1\right). 
    \end{eqnarray}
    {One can immediately see that this coincides with $\sigma_u^2$ given by \eqref{equ:vari_unimod} in Theorem \ref{thm:main_lattice}.}
        \item We give a detailed proof of Theorem \ref{thm:main_affine}. For the proof of Theorem   \ref{thm:main_congruence}, we omit some details, as they are analogous to the proof of Theorem \ref{thm:main_affine}. We indicate the necessary changes required in the proof of Theorem \ref{thm:main_congruence}.  
    \end{enumerate}
\end{remark}

\subsection{Outline of the proof of the above Theorems} To prove the above theorems, we closely follow the strategy of Bj$\ddot{\text{o}}$rklund and Gorodnik (\cite{BG1}). We provide a brief outline of the proof for Theorem \ref{thm:main_affine}, noting that the proof of Theorem \ref{thm:main_congruence} also follows a similar strategy. By an easy reduction, we can show that it is enough to prove Theorem \ref{thm:main_affine} for $T$ of the form $2^M,$ with $M \in \mathbb N$. If $T=2^M,$  $\Omega_T$ can be nicely tessellated by means of the action of a diagonal subgroup of $\SL_{m+n}(\mathbb R).$ To handle the case of $T=2^M,$ we use the method of cumulants and a powerful theorem of Fr\'{e}chet and Shohat (Theorem \ref{thm:CLT_crit}), which gives a criterion under which a sequence of bounded measurable functions on a probability space converge in distribution to the normal distribution. However, we cannot directly apply Theorem \ref{thm:CLT_crit} to the functions 
$$\Lambda \longmapsto \frac{|\Lambda \cap \Omega_T| - \vol(\Omega_T)}{\vol(\Omega_T)^{1/2}}, \,\,\text{where} \,\, T=2^M \,\, \text{with}\,\,
 M\in \mathbb N,$$
because these functions are typically unbounded. To address this issue, we exploit the tessellation property of our domain to observe that, for $T=2^M$, the conclusion of Theorem \ref{thm:main_affine} can be reformulated (see Theorem \ref{thm:affine_reduc}) in terms of $\SL_{m+n}(\mathbb R) \ltimes \mathbb R^{m+n}$-translates of the Siegel transform $\hat{\chi}$ of the indicator function $\chi$ of $\Omega_2$. Furthermore, we can approximate $\hat{\chi}$ by a family of smooth and bounded $C_c^{\infty}$ functions on the space of affine unimodular lattices of $\mathbb R^{m+n}$. We show that it is sufficient to prove the central limit theorem for $\SL_{m+n}(\mathbb R) \ltimes \mathbb R^{m+n}$-translates of the smooth approximation to $\hat{\chi}$ (see Proposition \ref{prop:H_epsilon_L}).  We then verify the CLT criteria of  Fr\'{e}chet and Shohat for these bounded functions. This amounts to showing that the variance is finite and that the cumulant of order $r$ ($r \geq 3$) for these functions tends to $0$ as $M \rightarrow \infty.$ In the computation of cumulants, the main tools are the quantitative multiple mixing estimate of \cite[Theorem 1.1]{BEG} and a combinatorial tool developed by Bj$\ddot{\text{o}}$rklund and Gorodnik in \cite[Proposition 6.2  ]{BG1'} to analyze the cumulants.\\

 \textbf{Acknowledgement :} The author is grateful to his advisor, Prof. Arijit Ganguly, for helpful discussions and continuous encouragement. Part of this work was done when the author was visiting the University of Z$\ddot{\text{u}}$rich to attend the Z$\ddot{\text{u}}$rich Dynamics School and the Z$\ddot{\text{u}}$rich Dynamics Conference in June $2023$. The author thanks the University of Z$\ddot{\text{u}}$rich for their generous hospitality. Thanks are also due to Prof. Micheal Bj$\ddot{\text{o}}$rklund and Prof. Alexander Gorodnik for helpful discussions during the Z$\ddot{\text{u}}$rich dynamics conference. The author would like to thank the referees for their very detailed reports, insightful comments, and many generous suggestions (in particular, the idea to improve the main results from $m+n \geq 5$ to $m+n \geq 4$), which have significantly improved the paper.

\section{Proof of the CLT for affine unimodular lattices}\label{sec:affine}
Siegel transform is one of the main tools in the study of lattice point counting problems. Throughout this section, we denote the set of all affine unimodular lattices on $\mathbb R^{m+n}$ by $X_a$, $\mu_a$ denotes the measure $\mu_{m+n,a}$ defined in the last section, and $l:=m+n$. Let $f : \mathbb R^l \rightarrow \mathbb R$ be a Borel measurable function with compact support. The Siegel transform of $f$ is $\hat{f}:X_{a} \rightarrow \mathbb R$ defined by
\begin{equation}\label{equ:21}
\hat{f}(\Lambda) :=\sum_{\lambda \in \Lambda} f(\lambda) \quad \text{for} \,\, \Lambda \in X_a.
\end{equation}
Let us record the analog of Siegel's mean value theorem and Rogers's second-order moment formula for affine lattices. We refer to \cite{A}, \cite{EMV}, and \cite{GH} for the proof of the following proposition.
\begin{proposition}\label{prop:Siegel_affine}
    For $f \in L^1(\mathbb R^l) \cap L^2(\mathbb R^l)$, we have the following:
    \begin{equation}\label{equ:22}
        \int_{X_a} \hat{f}(\Lambda) \,\, d\mu_a(\Lambda)=\int_{\mathbb R^l} f(\mathbf x) \,\, d\mathbf x
    \end{equation}
\begin{equation}\label{equ:23}
 \int_{X_a} {\hat{f}(\Lambda)}^2 \,\, d\mu_a(\Lambda)= \left(\int_{\mathbb R^l} f(\mathbf x) \,\, d\mathbf x \right)^2 + \int_{\mathbb R^l} {f(\mathbf x)}^2 \,\, d\mathbf x,
\end{equation}
where $d\mathbf x $ is the usual Lebesgue measure on $\mathbb R^l$  
\end{proposition}
Siegel transform of bounded functions are typically unbounded, and the following function has a close relationship with the growth of the Siegel transform.
\begin{definition}\label{defn:alpha}
   { Let $\Lambda$ be any given lattice of $\mathbb R^l$. We say that a subspace $V$ of $\mathbb R^l$ is $\Lambda$-rational if $V \cap \Lambda$ is a lattice in $V$. If $V$ is $\Lambda$-rational, we denote by $d_{\Lambda}(V)$ the volume of $V/(V \cap \Lambda).$ We define}
     \begin{equation}\label{equ:alpha_s}
{\alpha_l(\Lambda):=\sup \left\{d_{\Lambda}(V)^{-1}: V\cap \Lambda \,\, \text{is a $\Lambda$-rational subspace of} \,\, \mathbb R^l \right\}.}
\end{equation}
 
\end{definition}
An immediate consequence of Mahler's compactness criterion shows that $\alpha_l$ is a proper map on $X_l.$  We define $\alpha_{l,a} : X_{a} \rightarrow \mathbb R$ by
\begin{equation}\label{equ:alpha_affine}
     \alpha_{l,a}(\Lambda)= \alpha_{l+1}(\Tilde{h}(\Lambda))=\alpha_l(\Tilde{\pi}(\Lambda)),
\end{equation}
where $\alpha_l$ is given by \eqref{equ:alpha_s} and $\tilde{h}$, $\tilde{\pi}$ are defined in Subsection \ref{subsec:affine}. A routine verification shows that the above definition indeed makes sense {(for details see \cite[\S 3.3.1]{AG}).}

Throughout, given $A,B \in \mathbb R$, $A \ll_{z} B$ means that there exists $c>0$ (depending only on $z$) such that $A \leq cB.$ Let us recall the following propositions, which will be useful later.

 
\begin{proposition}{\cite[Proposition 3.14]{AG}}\label{prop:Siegel_alpha_bound_affine}
    Let $g : \mathbb R^l \rightarrow \mathbb R$ be a bounded function with compact support. Then
    \begin{equation*}
        |\hat{g}(\Lambda)| \ll_{\supp (g)} \|g\|_{\infty} \cdot \alpha_{l,a}(\Lambda) \quad \forall \Lambda \in X_{a},
    \end{equation*}
    where $\supp (g)$ denotes the support of $g.$
\end{proposition}

\begin{proposition}{\cite[Proposition 3.17]{AG}}\label{prop:alpha_Lp_affine}
     Let $\alpha_{l,a}$ be the function defined on $X_{a}$ by \eqref{equ:alpha_affine}. Then $\alpha_{l,a} \in L^p(X_{a},\mu_a)$ for $1 \leq p < l$ and for any $L>0,$
    \begin{equation}\label{equ:alpha_Lp_bound_affine}
\mu_a (\{\alpha_{l,a}  \geq L\}) \ll_p L^{-p} \quad \forall \,\, p <l.
\end{equation}
\end{proposition} 
Now we show that it is enough to prove Theorem \ref{thm:main_affine} for the case $T=2^M$ with $M \in \mathbb N$.
\begin{proposition}\label{prop:reduction_2M}
Suppose that Theorem \ref{thm:main_affine} holds true for $T$ of the form $2^M$ with $M \in \mathbb N$. Then Theorem \ref{thm:main_affine} holds true in its full generality.
\begin{proof}
    First, let us compute the volume of $\Omega_T.$ Let $ d \mathbf x $ and $d \mathbf y$ be the usual Lebesgue measure on $ \mathbb R^m$ and $\mathbb R^n$ respectively. Then
    \begin{eqnarray}\label{equ:omega_T_vol_p}
    \vol(\Omega_T) & = & {\int \int}_{\Omega_T} d\mathbf x \,\,  d  \mathbf y \nonumber \\ & = & \int_{1 \leq \|\mathbf y\| \leq T} \left(\int_{-\frac{c_1}{{\|\mathbf y\|}^{u_1}}}^{\frac{c_1}{{\|\mathbf y\|}^{u_1}}} \dots \int_{-\frac{c_m}{{\|\mathbf y\|}^{u_m}}}^{\frac{c_m}{{\|\mathbf y\|}^{u_m}}} dx_1 \dots dx_m \right) d\mathbf{y} \nonumber \\ &=& \int_{1 \leq \|\mathbf y\| \leq T} \frac{2^mc_1 \dots c_m}{\|\mathbf y\|^n}  d\mathbf y  
    \end{eqnarray}
Now we change the Cartesian coordinates $\mathbf{y}=(y_1,\dots,y_n)$ of $\mathbb R^n$ to spherical coordinates $(r,\theta_1,\dots,\theta_{n-1}).$ Let $r^{n-1} S(\theta_1,\dots,\theta_{n-1})drd\theta_1 \dots d\theta_{n-1}$ be the spherical volume element on $\mathbb R^n$ and
\begin{equation}\label{equ:omega_n_main}
    \omega_n =\int_{0}^{2\pi} \int_{0}^{\pi} \dots \int_{0}^{\pi} S(\theta_1,\dots,\theta_{n-1}) \,\, d\theta_1 \dots d\theta_{n-1}.
\end{equation}
Then, in view of the above coordinate change, we have
\begin{eqnarray}\label{equ:omega_T_vol}
    \vol(\Omega_T)  &=& 2^mc_1 \dots c_m \omega_n \int_{1}^{T} \frac{ 1}{r^n}r^{n-1}\,\, dr \nonumber \\ &=& 2^mc_1 \dots c_m \omega_n \log T.
    \end{eqnarray}

\noindent Given any real parameter $T>0$ there exists $M=M(T)\in \mathbb Z_{\geq 0}$ such that $\frac{T}{2} < 2^M \leq T \leq 2^{M+1}.$ Now for any affine lattice $\Lambda \in X_{m+n,a},$ we have 
\begin{eqnarray}
\frac{|\Lambda \cap \Omega_T|- \vol(\Omega_T)}{\vol(\Omega_T)^{1/2}} &=& \frac{|\Lambda \cap \Omega_{2^M}|- \vol(\Omega_{2^M})}{\vol(\Omega_T)^{1/2}} +\frac{|\Lambda \cap (\Omega_T \setminus \Omega_{2^M})|}{\vol(\Omega_T)^{1/2}}- \frac{\vol(\Omega_T \setminus \Omega_{2^M})}{\vol(\Omega_T)^{1/2}} \\ \nonumber &=& a_T X_T + Y_T - b_T,
\end{eqnarray}
 where 
 \begin{equation}
     a_T=\left(\frac{\vol(\Omega_{2^M})}{\vol(\Omega_T)} \right)^{1/2}, X_T=\frac{|\Lambda \cap \Omega_{2^M}|- \vol(\Omega_{2^M})}{\vol(\Omega_{2^M})^{1/2}}, Y_T=\frac{|\Lambda \cap (\Omega_T \setminus \Omega_{2^M})|}{\vol(\Omega_T)^{1/2}}, b_T= \frac{\vol(\Omega_T \setminus \Omega_{2^M})}{\vol(\Omega_T)^{1/2}}.
 \end{equation}
By assumption $X_T$ converges in distribution to the normal distribution $N(0,1)$ as $T \rightarrow \infty.$ Hence in order to show  \eqref{equ:affine_main_conv}, it is enough to show that {$a_T \rightarrow 1,$ $b_T \rightarrow 0$ and  $Y_T \rightarrow 0$ in $L^1$ (which also implies convergence in distribution) as $T \rightarrow \infty.$}

\noindent Note that $a_T=\left(\frac{M \log 2}{\log T} \right)^{1/2}$ and also we have 
\begin{equation*}
     \begin{array}{lcl}
     \displaystyle  \quad \quad \frac{T}{2} < 2^M \leq T \\
        \displaystyle   \implies  \log T-\log 2 < M \log 2 \leq \log T \\  \displaystyle \implies \left(1-\frac{\log 2}{\log T} \right)^{1/2} < a_T \leq 1 \\ \displaystyle \implies a_T \rightarrow 1 \,\, \text{as} \,\, T \rightarrow \infty.
         \end{array}
 \end{equation*}
{Using the fact that $T<2^{M+1}$, we get 
\begin{equation}\label{equ:vol_T_2_M}
    \vol(\Omega_T \setminus \Omega_{2^M})=\gamma_{m,n} \left(\log T -\log 2^{M} \right) \leq \gamma_{m,n} \log 2,
\end{equation} 
where $\gamma_{m,n}=2^mc_1 \dots c_m \omega_n$ is the constant appeared in the expression of volume of $\Omega_T$.}

\noindent {Therefore, using Siegel's integral formula (Proposition \ref{prop:Siegel_affine}) and \eqref{equ:vol_T_2_M}, we get
\[0<b_T \leq \gamma_{m,n} \frac{\log 2}{(\log T)^{1/2}} \rightarrow 0 \,\, \text{as} \,\, T \rightarrow \infty \]}

\noindent {and \[ \int_{X_a} |Y_T (\Lambda)| \, d\mu_a(\Lambda)=\frac{\vol(\Omega_T \setminus \Omega_{2^M})}{\vol(\Omega_T)^{1/2}} \leq \frac{\gamma_{m,n} \log 2}{\vol(\Omega_T)^{1/2}}\rightarrow 0 \,\, \text{as} \,\, T \rightarrow \infty.\]}

Hence, we are done.
 
\end{proof}
\end{proposition}

Now we reformulate our problem using the Siegel transform. Let $\chi_{\Omega_T}$ denotes the indicator function of $\Omega_T.$ For $T=2^M$ and $\Lambda \in X_a$, we have
\begin{equation}\label{equ:24}
|\Lambda \cap \Omega_{2^M}|=\hat{\chi}_{\Omega_{2^M}}(\Lambda)  \quad \text{and } \quad \vol(\Omega_{2^M})=\int_{X_a} \hat{\chi}_{\Omega_{2^M}}(\Lambda) \,\, d\mu_a(\Lambda)
\end{equation}
Note that our domain $\Omega_{2^M}$ can be tessellated by the action of a diagonal element of $\SL_l(\mathbb R).$ Let
\begin{equation}\label{equ:25}
    {c}_0=\diag (2^{u_1},\dots,2^{u_m},2^{-1},\dots,2^{-1})
\end{equation}
and 
\begin{equation}
     c=(\diag (2^{u_1},\dots,2^{u_m},2^{-1},\dots,2^{-1}),0) \in \SL_l(\mathbb R)  \ltimes \mathbb R^n.
\end{equation}
Then
\begin{equation}\label{equ:26}
\Omega_{2^M}=\bigsqcup_{k=0}^{M-1} {c}_0^{-k} \Omega_2.
\end{equation}
Hence
\begin{equation}\label{equ:27}
|\Lambda  \cap \Omega_{2^M} |=\sum_{k=0}^{M-1} \hat{\chi}_{\Omega_2}(c^k \Lambda)
\end{equation}
and 
\begin{equation}\label{equ:28}
\vol(\Omega_{2^M})=\int_{X_a} \sum_{k=0}^{M-1} \hat{\chi}_{\Omega_2}(c^k \Lambda) \,\, d\mu_a(\Lambda)=\sum_{k=0}^{M-1} \vol(\Omega_2)=M \cdot \vol(\Omega_2 )
\end{equation}
From now on we will denote $\chi_{\Omega_2}$ by $\chi$. From the above discussion, it is clear that the following Theorem is equivalent to Theorem \ref{thm:main_affine}. 
\begin{theorem}\label{thm:affine_reduc}
Under the hypothesis of Theorem \ref{thm:main_affine},
       
\begin{equation}\label{equ:29}
H_M:=\frac{1}{\sqrt{M}} \left( \sum_{k=0}^{M-1}  \hat{\chi} \circ c^k - M \cdot \vol(\Omega_2)\right) \longrightarrow N(0,\eta^2) \,\, \text{in distribution as} \,\, M \rightarrow \infty,
\end{equation}
    where $\eta=\vol(\Omega_2)^{1/2}$ and $N(0,\eta^2)$ denotes the normal distribution with mean $0$ and variance $\eta^2.$
\end{theorem}
We prove Theorem \ref{thm:affine_reduc} using the CLT criteria of Fr\'{e}chet and Shohat. Let us first recall the notion of a \emph{cumulant}.

\begin{definition}\label{def:cumu}
    Let $(Y, \nu)$ be a probability space. For $\psi_1,\dots,\psi_r \in L^{\infty}(Y),$ their \emph{joint cumulant} of order $r$ is defined by
\begin{equation*}
\text{Cum}_{[r]}(\psi_1,\dots,\psi_r)=\sum_{\mathcal P}(-1)^{ |\mathcal P| -1} \left( |\mathcal P|-1 \right)! \prod_{I \in \mathcal P} \int_{Y} \prod_{i \in I} \psi_i \,\, d\nu,
\end{equation*}
where $\mathcal P$ is the set of all partition of the set $\{1,\dots,r\}.$
\end{definition}
 For a partition $\mathcal Q$ of $\{1,\dots,r\}$, the  \emph{conditional joint cumulant} of $\psi_1,\dots,\psi_r \in L^{\infty}(Y)$ with respect to $\mathcal Q$ is defined as
 \begin{equation*}
    \text{Cum}_{[r]}(\psi_1,\dots,\psi_r| \mathcal Q)=\sum_{\mathcal P}(-1)^{ |\mathcal P| -1} \left(|\mathcal P|-1 \right)! \prod_{I \in \mathcal P} \prod_{J \in \mathcal Q} \int_{Y} \prod_{i \in I \cap J} \psi_i \,\, d\nu.
 \end{equation*}
 For $\psi \in L^{\infty}(Y),$ we define 
 \begin{equation*}
     \text{Cum}_{[r]}(\psi):= \text{Cum}_{[r]}(\psi,\dots,\psi).
 \end{equation*}
 Note that $\text{Cum}_{[r]}$ is  multi linear in the functions $\psi_1,\dots,\psi_r.$ {We now state the CLT criteria of Fr\'{e}chet and Shohat (\cite{FS}) appeared in \cite[Proposition 3.4]{BG1} in the following form.}
 \begin{theorem}\label{thm:CLT_crit}
     Let $(Y,\nu)$ be a probability space and $\{\Psi_M\}$ be a sequence of real-valued measurable bounded functions on $Y$ satisfying
     \begin{equation}\label{equ:FS_1}
         \int_Y \Psi_M \,\, d\nu =0 \quad \text{and} \quad \sigma^2 :=\lim_{M \rightarrow \infty} \int_{Y} \Psi_M^2 < \infty
         \end{equation}
         and 
 \begin{equation}\label{equ:FS_2}
     \lim_{M \rightarrow \infty} \text{Cum}_{[r]} (\Psi_M)=0 \quad \text{for all } \,\, r\geq 3.
 \end{equation}
 Then for every $\xi \in \mathbb R$
 \begin{equation*}
     \nu(\{\Psi_M < \xi\}) \rightarrow {N}_{\xi}(0,\sigma^2) \quad \text{as} \quad M \rightarrow \infty.
 \end{equation*}
 \end{theorem}
  
\begin{proposition}{{\cite[Proposition 8.1]{BG1'}}}\label{prop:condi_cumu}
    Let $\psi_1,\dots,\psi_r \in L^{\infty}(Y,\mu)$ and $\mathcal Q$ be any partition of $\{1,\dots,r\}$ with $ |\mathcal Q| \geq 2.$ Then we have
    \begin{equation*}
        \text{Cum}_{[r]}(\psi_1,\dots,\psi_r | \mathcal Q)=0.
    \end{equation*}
\end{proposition}

Before proceeding further, we recall the notion of operator norm on {$H_{a}:=\SL_l(\mathbb R) \ltimes \{0\} $ and the Sobolev norms on the space $C_c^{\infty}(X_{a}).$ Note that  
  $H_{a}$ is a connected semisimple Lie group with finite centre and $K_a:=\SO(n)$ is the maximal compact subgroup of $H_a$. Let  $\mathfrak{h}_a:=\Lie(H_a)$ be the Lie algebra of $H_a$ and $\langle \cdot, \cdot \rangle$ be an $\Ad(K_a)$-invariant inner product on $\mathfrak{h}_a$. Also let $\|\cdot\|$ be the corresponding norm on $\mathfrak{h}_a$. Suppose that $\rho_{H_{a}}$ be the right-invariant distance function on $H_{a}$ induced by the Riemannian metric corresponding to $\langle \cdot, \cdot \rangle$.}
       \begin{definition}\label{defn:operator_norm}
          We define the submultiplicative operator norm $\|\cdot\|_{\text{op}}$ on $H_{a}$ by
      \begin{equation}
          \|h\|_{\text{op}} := \max \{\|\Ad(h)Y\|: Y \in \mathfrak{h}_a,\,\, \|Y\|=1\} \quad \forall h \in H_{a}.
      \end{equation}
       \end{definition}
  
      In other words, the operator norm of $h$ is just the operator norm of the adjoint map $\Ad(h): \mathfrak{h}_a \rightarrow \mathfrak{h}_a.$ Now we record the following lemma, which will be useful later. The proof of the following lemma can be found in \cite[Lemma 2.1]{BEG}.
      \begin{lemma}\label{lem:operator_norm}$\ $
          \begin{enumerate}[label=({{\roman*}})]
              \item For all $h \in H_{a},$ $\exists$ constants $c_1 \geq 1$ and $c_2 >0$ such that
              \begin{equation}
            c_1^{-1} \log \|h\|_{\text{op}} -c_2\leq \rho_{H_{a}}(h,e) \leq c_1 \log \|h\|_{\text{op}} +c_2,
              \end{equation}
              where $e$ is the identity element of $H_{a}.$
              \item Let $c=(\diag (2^{u_1},\dots,2^{u_m},2^{-1},\dots,2^{-1}),0) \in H_{a}.$ Then $\exists$ $\lambda >1$ such that for any $q \in \mathbb N,$
              \begin{equation*}
                  \|c^q\|_{\text{op}} \geq \lambda^q.
              \end{equation*}
        \end{enumerate}
      \end{lemma}

Let $\Lie(G_{a})$ be the Lie algebra of $G_a$. Every $Z \in \Lie(G_{a})$ gives rise to a first order differential operator $\mathcal{D}_Z$ on $C_c^{\infty}(X_{a})$ defined by 
\begin{equation*}
    \mathcal{D}_Z(\psi)(x):= \lim_{t \rightarrow 0} \frac{\psi(\exp(tZ)x)-\psi(x)}{t} \quad \forall \psi \in C_c^{\infty}(X_{a}).
\end{equation*}
If $\{Z_1,\dots,Z_r\}$ be a fix ordered basis of $\Lie(G_{a}),$ then every element of the universal enveloping algebra $\mathcal{U}(\Lie(G_{a}))$ of $\Lie(G_{a})$ can be written  as a linear combination of monomials $Z_1^{n_1}\dots Z_r^{n_r}$ in the basis elements. Hence we extend the differential operator to $\mathcal{U}(\Lie(G_{a}))$ by defining it for monomials $Y=Z_1^{n_1}\dots Z_r^{n_r}$ as
\begin{equation*}
    \mathcal{D}_Y=\mathcal{D}_{Z_1}^{n_1} \dots \mathcal{D}_{Z_r}^{n_r}.
\end{equation*}
The degree of $\mathcal{D}_Y$ is defined as $\deg(Y):=n_1 + \dots+n_r.$ Now we define a family of Sobolev norms on $C_c^{\infty}(X_{a}).$
\begin{definition}\label{defn:sobolev_norm}
    Let $k \in \mathbb N$ and $\psi \in C_c^{\infty}(X_{a}).$ We define the following
    \begin{equation*}
        \|\psi\|_{C^k}:= \max \{\|\mathcal{D}_Y \psi\|_{\infty}: Y \in \mathcal{U}(\Lie(G_{a})) \,\, \text{is a monomial,} \,\, \deg(Y) \leq k \},
    \end{equation*}
    \begin{equation*}
        S_k(\psi):=\max \{\|\psi\|_{\infty},\|\psi\|_{C^k}\}.
    \end{equation*}
\end{definition}

\begin{lemma}\label{lem:pro_sob_norm} 
The family of Sobolev norms on $C_c^{\infty}(X_{a})$ satisfies the following properties.
    \begin{enumerate}[label=({{\arabic*}})]
        \item For any $\psi \in C_c^{\infty}(X_{a})$ and $g \in G_{a}$,
        \begin{equation*}
         S_k(\psi \circ g ) \ll_k \|g\|_{\text{op}}^k S_k(\psi).
        \end{equation*}
        \item For any $\psi_1,\psi_2 \in C_c^{\infty}(X_{a}),$
        \begin{equation*}
             S_k(\psi_1 \psi_2) \ll_k S_k(\psi_1) S_k(\psi_2).
        \end{equation*}
\end{enumerate}
\end{lemma}
\begin{proof}
    The proof of these properties of Sobolev norms is very standard. Hence, we leave the details to interested readers.
\end{proof}

Next, we recall the quantitative exponential mixing of all orders for Lie groups from \cite{BEG}, which will be very instrumental in the computation of cumulants later. {Note that $H_a$ is a connected semisimple Lie subgroup of $G_a$ with finite centre, and the action of $H_a$ on $X_a$ has strong spectral gap (for the proof see \cite[Theorem 4.5]{KW}). In fact, this is immediate, as $H_a$ has Kazhdan Property $(T)$ (since $l=m+n\geq 4$, we are in a higher-rank situation). The following is a particular case of \cite[Theorem 1.1]{BEG}.}
\begin{theorem}\label{thm:quan_mixing}
    For all integer $r\geq 2$ and all sufficiently large $q,$ there exists $\delta=\delta(r,q)>0$ such that for all $\psi_1,\dots,\psi_r \in C_c^{\infty}(X_a)$ and $h_1,\dots,h_r \in H_{a}=\SL_l(\mathbb R) \ltimes \{0\}$
    \begin{equation}
        \left|\int_{X_a} (\psi_1 \circ h_1) \dots (\psi_r \circ h_r) \,\, d\mu_a - \prod_{i=1}^{r} \left(\int_{X_a} \psi_i \,\, d\mu_a\right)\right| \ll_{r,q} e^{-\delta \min_{i\neq j} \rho_{H_{a}}(h_i,h_j)} S_q(\psi_1) \dots S_q(\psi_r).
    \end{equation}
\end{theorem}

\begin{remark}\label{rem:quan_mix_cong}
    {Since for $m+n \geq 3,$ $\SL_{m+n}(\mathbb R)$ has Kazhdan Property $(T)$, the action of $\SL_{m+n}(\mathbb R)$ on $X_{m+n,c}$, the space of all congruence lattices, has strong spectral gap. Therefore, a quantitative multiple mixing estimate analogous to Theorem \ref{thm:quan_mixing} also holds in the congruence setup.}
\end{remark}

To apply Theorem \ref{thm:quan_mixing} in the computation of cumulants, we need to approximate $\hat{\chi}$ by smooth functions. For that first we approximate $\chi$ (following \cite[Section 6]{BG1})by a family of non-negative smooth functions $f_{\varepsilon} \in C_{c}^{\infty}(\mathbb R^{m+n})$ such that $\text{supp }$$f_{\varepsilon} \subseteq (\Omega_2)_{\varepsilon}$, and
\begin{equation}\label{equ:chi_f_epsilon}
    \chi \leq f_{\varepsilon} \leq 1,\,\, \|f_{\varepsilon}\|_{C^k} \ll \varepsilon^{-k},\,\, \|\chi - f_{\varepsilon}\|_{1} \ll \varepsilon  , \,\, \| \chi - f_{\varepsilon}\|_2 \ll \sqrt{\varepsilon},
\end{equation}
where $(\Omega_2)_{\varepsilon}$ is an $\varepsilon$-neighbourhood of the set $\Omega_2.$

For simplicity, from now on by $\alpha$ we denote the function $\alpha_{a,l}$ on $X_{a}$ defined by \eqref{equ:alpha_affine}.
\begin{lemma}\label{lem:trun_Siegel}
    For every $b>0$, there is a family of smooth functions $\{\eta_L\} \in C_{c}^{\infty}(X_a)$ such that, for every $L>0$
    \begin{equation}\label{equ:trun_Siegel}
        0 \leq \eta_L \leq 1, \,\, \eta_L=1 \,\, \text{on}\,\, \{\alpha \leq b^{-1}L\},\,\, \eta_L=0 \,\, \text{on}\,\, \{\alpha > bL\},\,\, \|\eta_L\|_{C^k} \ll 1.
    \end{equation}
\end{lemma}
\begin{proof}
    The proof of this Lemma is analogous to \cite[Lemma 4.11]{BG1}.
\end{proof}

The truncated Siegel transform of a bounded function $f : \mathbb R^{m+n} \rightarrow \mathbb R$ with compact support is defined as
$$\hat{f}^{(L)}:=\hat{f} \cdot \eta_L.$$
\begin{lemma}\label{lem:tranc_Siegel_property}
If $f \in C^{\infty}_{c}(\mathbb R^l),$ then the truncated Siegel transform $\hat{f}^{(L)} \in  C^{\infty}_{c}(X_a) $ and also we have  
\begin{equation}
    \left\|\hat{f}^{(L)}\right\|_{\infty} \ll_{\supp(f)} L \|f\|_{\infty},
\end{equation}
\begin{equation}
    \left\|\hat{f}^{(L)}\right\|_{C^k} \ll_{\supp(f)} L \|f\|_{C^k},
\end{equation}
\begin{equation}
 \left\|\hat{f} - \hat{f}^{(L)}\right\|_1  \ll_{\supp(f),p} L^{-p} \|f\|_{\infty} \quad  \forall \,\, p < m+n-1 ,  
\end{equation}
\begin{equation}
     \left\|\hat{f} - \hat{f}^{(L)}\right\|_2  \ll_{\supp(f),p} L^{-(p-1)/2} \|f\|_{\infty} \quad  \forall \,\, p < m+n-1   
\end{equation}
\begin{equation}
S_k \left(\hat{f_{\varepsilon}}^{(L)} \right) \ll_{k} \varepsilon^{-k} L,
\end{equation}
where $f_{\varepsilon}$ is given by \eqref{equ:chi_f_epsilon}.
\end{lemma}
\begin{proof}
The proof of the first four inequalities follows from \cite[Lemma 4.12]{BG1}. For the last inequality, consider
\begin{eqnarray}
S_k \left(\hat{f_{\varepsilon}}^{(L)}\right) &=& \max \left\{\left\|\hat{f_{\varepsilon}}^{(L)}\right\|_{\infty},\left\|\hat{f_{\varepsilon}}^{(L)}\right\|_{C^k} \right\} \\ &\ll_{\supp(f_{\varepsilon})}& \max \left\{L\|f_{\varepsilon}\|_{\infty},L\|f_{\varepsilon}\|_{C^k} \right\} \\ &\ll& \max \{L,L\varepsilon^{-k}\}, \,\, \text{using} \,\, \eqref{equ:chi_f_epsilon} \\ &=& \varepsilon^{-k}L.
\end{eqnarray}
\begin{remark}\label{rem:trun_sei_con}
    {The analogous estimates involving the truncated Siegel transform as in Lemma \ref{lem:tranc_Siegel_property} also hold in the congruence setup (for example, see \cite[Lemma 4.15]{AG} )}
\end{remark}
   
\end{proof}
{Recall that \[H_M:=\frac{1}{\sqrt{M}} \left( \sum_{k=0}^{M-1}  \hat{\chi} \circ c^k - M \cdot \vol(\Omega_2)\right).\]}
Now we are going to approximate $\hat{\chi}$ by means of truncated Siegel transforms of $f_{\varepsilon},$ where $f_{\varepsilon}$ are given by \eqref{equ:chi_f_epsilon}. Next we approximate $H_M$ by $H_M^{(\varepsilon,L)}$ such that $\|H_M-H_M^{(\varepsilon,L)}\|_1 \rightarrow 0.$ Then $H_M$ and $H_M^{(\varepsilon,L)}$ converge to the same limit in distribution. Due to the approximation, we get two sequences $\varepsilon(M)$ and $L(M)$ which depends on $M.$ At the end, we give explicit choices of these sequences.
\begin{proposition}\label{prop:H_epsilon_L}
   Define 
   \begin{equation}\label{equ:H_epsilon_L}
       H_M^{(\varepsilon,L)}:=\frac{1}{\sqrt{M}} \left( \sum_{k=0}^{M-1}   \hat{f_{\varepsilon}}^{(L)} \circ c^k - M \int_{X_a} \hat{f_{\varepsilon}}^{(L)} \,\,  d\mu_a \right),
   \end{equation}
   where $\hat{f_{\varepsilon}}^{(L)}=\hat{f_{\varepsilon}} \eta_L \in C_{c}^{\infty}(X_a)$ are smooth approximations to $\hat{\chi}.$
   Then, to prove  Theorem \ref{thm:affine_reduc} it is enough to show that 
   \begin{equation*}
        H_M^{(\varepsilon,L)} \rightarrow N(0,\eta^2) \,\, \text{in distribution as} \,\, M \rightarrow \infty,
    \end{equation*}
    {provided $ \varepsilon =   o(M^{-1/2}) \text{ and } M=o(L^{l-2}).$} Here $\eta=\vol(\Omega_2)^{1/2}$.
\end{proposition}
\begin{proof}
    First, note that to prove this proposition, it is enough to show that 
    $$ \|H_M^{(\varepsilon,L)} - H_M\|_1 \rightarrow 0  \quad \text{as} \,\, M \rightarrow \infty.$$
    \begin{eqnarray} 
        \|H_M^{(\varepsilon,L)} - H_M\|_1 &=& \left\|\frac{1}{\sqrt{M}} \left( \sum_{k=0}^{M-1}   \left(\hat{f_{\varepsilon}}^{(L)} -\hat{\chi}\right)\circ c^k - M \int_{X_a} \left(\hat{f_{\varepsilon}}^{(L)} - \hat{\chi}\right) \,\,  d\mu_a \right)\right\|_1 \nonumber \\ &\leq& \frac{1}{\sqrt{M}} \sum_{k=0}^{M-1} \left\|\left(\hat{f_{\varepsilon}}^{(L)} -\hat{\chi}\right)\circ c^k \right\|_1 + \sqrt{M} \left\| \hat{f_{\varepsilon}}^{(L)} -\hat{\chi} \right\|_1 \nonumber \\ & =  & 2 \sqrt{M} \left\|   \hat{f_{\varepsilon}}^{(L)} -\hat{\chi}  \right\|_1, \,\, \text{by using}\,\, G_{a}\text{-invariance of} \,\, \mu \nonumber \\ & \leq & 2\sqrt{M} \left(\left\| \hat{f_{\varepsilon}} - \hat{f_{\varepsilon}}^{(L)} \right\|_1 + 
  \left\|\hat{f_{\varepsilon}} - \hat{\chi}\right\|_1 \right) \nonumber \\ &\ll& \sqrt{M}(L^{-(l-2)/2} + \varepsilon),
    \end{eqnarray}
by using \eqref{equ:chi_f_epsilon}, Lemma \ref{lem:tranc_Siegel_property}, the fact that $f_{\varepsilon} \leq 1$ and the family $\{\supp \,\, f_{\varepsilon}\}_{\varepsilon}$ is uniformly bounded.  

To get our desired result we choose the parameters $\varepsilon$ and  $L$ as functions of $M$ such that as $M \rightarrow \infty$,
\begin{equation}\label{equ:epsilon_L}
 \varepsilon =   o(M^{-1/2}) \quad \text{and} \quad M=o(L^{l-2}).
\end{equation}

\end{proof}

\begin{lemma}\label{lem:f_epsilon_L_chi_hat_L_2}
\begin{enumerate}[label=({{\roman*}})]
    \item {$\left\| \hat{f_{\varepsilon}}^{(L)} -\hat{\chi} \right\|_2 \ll L^{-(l-3)/2}+ \sqrt{\varepsilon}.$}
    \item \begin{equation}
        {\left| \int_{X_a} (\psi \circ c^k)\psi \, d\mu_a-\int_{X_a} (\phi \circ c^k)\phi \, d\mu_a\right| \ll  L^{-(l-3)/2}+ \sqrt{\varepsilon},}
    \end{equation}  where 
    
    \begin{equation}
        {\psi=\hat{f_{\varepsilon}}^{(L)} -\int_{X_a} \hat{f_{\varepsilon}}^{(L)} \,\,  d\mu_a \quad \text{and}\quad \phi:=\hat{\chi} -\vol(\Omega_2)}
    \end{equation} 
\end{enumerate} 
  
\end{lemma}
\begin{proof} \begin{enumerate}[label=({{\roman*}})]
    \item Using Rogers's second moment formula, Lemma \ref{lem:tranc_Siegel_property} and \eqref{equ:chi_f_epsilon}, we get 
    
    \begin{eqnarray}
       \left\| \hat{f_{\varepsilon}}^{(L)} -\hat{\chi} \right\|_2 &\ll&      \left\| \hat{f_{\varepsilon}} - \hat{f_{\varepsilon}}^{(L)} \right\|_2 + 
  \left\|\hat{f_{\varepsilon}} - \hat{\chi}\right\|_2   \nonumber \\  &\ll& \left\| \hat{f_{\varepsilon}} - \hat{f_{\varepsilon}}^{(L)} \right\|_2+ \left( \|f_{\varepsilon}-\chi\|_1^2+\|f_{\varepsilon}-\chi\|_2^2\right)^{1/2}\nonumber \\ &\ll& L^{-(l-3)/2}  +\left(\varepsilon^2+\sqrt{\varepsilon}^{2} \right)^{1/2} \nonumber \\ &\ll & L^{-(l-3)/2} +\sqrt{\varepsilon}.
\end{eqnarray}
\item First note that 
\[\psi -\phi=\left(\hat{f_{\varepsilon}}^{(L)}-\hat{\chi}\right) -\int_{X_a} \left(\hat{f_{\varepsilon}}^{(L)}-\hat{\chi}\right) \,\,  d\mu_a.\]
Since centering is an othogonal projection in $L^2$, we have 
\begin{equation}\label{equ:centring_L_2}
    \|\psi-\phi\|_2 \leq \|\hat{f_{\varepsilon}}^{(L)}-\hat{\chi}\|_2.
\end{equation}
\noindent Now 
\begin{eqnarray}
    \int_{X_a} (\psi \circ c^k)\psi \, d\mu_a &=& \int_{X_a} \left((\phi+(\psi-\phi)) \circ c^k\right)(\phi+(\psi-\phi)) \, d\mu_a \nonumber \\ &=& \int_{X_a} (\phi \circ c^k)\phi \, d\mu_a + \int_{X_a} ((\psi-\phi) \circ c^k)\phi \, d\mu_a \nonumber \\ &+& \int_{X_a} (\phi \circ c^k)(\psi-\phi) \, d\mu_a + \int_{X_a} ((\psi-\phi) \circ c^k)(\psi-\phi) \, d\mu_a
\end{eqnarray}
Therefore, by the Cauchy-Schwarz inequality and the invariance of the measure $\mu_a$, we get that
\begin{eqnarray}
   \left| \int_{X_a} (\psi \circ c^k)\psi \, d\mu_a -\int_{X_a} (\phi \circ c^k)\phi \, d\mu_a  \right| &\leq& \|(\psi-\phi) \circ c^k\|_2 \|\phi\|_2 +\|\phi \circ c^k\|_2 \|\psi-\phi\|_2 \nonumber \\ &+& \|(\psi-\phi) \circ c^k\|_2 \|\psi-\phi\|_2 \nonumber \\ &=& 2\|\psi-\phi\|_2 \|\phi\|_2  + \|\psi-\phi\|_2^2 \nonumber \\ &\ll & L^{-(l-3)/2} +\sqrt{\varepsilon},
\end{eqnarray}
where in the last line we are using part $(i)$, \eqref{equ:centring_L_2}. 
\end{enumerate} 
         
\end{proof}

In view of the Proposition \ref{prop:H_epsilon_L}, Theorem \ref{thm:affine_reduc} (hence Theorem \ref{thm:main_affine}) is equivalent to showing central limit theorem for  $H_M^{(\varepsilon,L)}$.
\begin{theorem}\label{thm:H_epsilon_L_CLT}
    {Suppose that the parameters $\varepsilon$ and $L$ satisfy
    \begin{equation}\label{equ:para_final}
        \varepsilon = M^{-2} \quad \text{and}\quad L=M^{1/(l-2)} \log M.
    \end{equation}}
      Then under the hypothesis of Theorem \ref{thm:main_affine},
    \begin{equation*}
        H_M^{(\varepsilon,L)} \rightarrow N(0,\eta^2) \,\, \text{in distribution as} \,\, M \rightarrow \infty,
    \end{equation*}
    where $\eta=\vol(\Omega_2)^{1/2}$.
\end{theorem}
To prove Theorem \ref{thm:H_epsilon_L_CLT}, we use the {CLT criteria of Theorem \ref{thm:CLT_crit}.} {We emphasize that our proof closely follows the strategy of Bj$\ddot{\text{o}}$rklund-Gorodnik (\cite{BG1}) and Aggarwal-Ghosh (\cite{AG})}.  By using Theorem \ref{thm:CLT_crit}, we immediately get Theorem \ref{thm:H_epsilon_L_CLT} if we can show that for \eqref{equ:para_final} as the choice of parameters $\varepsilon$ and $L$, we have 

\begin{equation}\label{equ:affine_CLT_condition_var}
\sigma^2 :=\lim_{M \rightarrow \infty} \|H_M^{(\varepsilon,L)}\|_2^2 < \infty
\end{equation}

\begin{equation}\label{equ:affine_CLT_condition_cum}
\lim_{M \rightarrow \infty} \text{Cum}_{[r]} \left(H_M^{(\varepsilon,L)} \right)=0 \quad \text{for all} \,\, r \geq 3.
\end{equation}

{\noindent \textbf{Computation of variance:} }

\noindent Let \[\psi=\hat{f_{\varepsilon}}^{(L)} -\int_{X_a} \hat{f_{\varepsilon}}^{(L)} \,\,  d\mu_a.\]
Then 
\[H_M^{(\varepsilon,L)}=\frac{1}{\sqrt{M}}\sum_{k=0}^{M-1}\psi \circ c^k.\]
Therefore,
\begin{eqnarray} 
    \|H_M^{(\varepsilon,L)}\|_2^2 &=&\frac{1}{M} \sum_{k_1,k_2=0}^{M-1} \int_{X_a} (\psi \circ c^{k_1}) (\psi \circ c^{k_2}) \,\, d\mu_a \nonumber \\ &=& \frac{1}{M} \sum_{k_1,k_2=0}^{M-1} \int_{X_a} (\psi \circ c^{k_1-k_2}) \psi \,\, d\mu_a \nonumber \\ & = &  \frac{1}{M} \sum_{\pm k=0}^{M-1} (M -|k|) \int_{X_a} (\psi \circ c^{k}) \psi \,\, d\mu_a \nonumber \\ &=&   \sum_{|k| <M }  \left(1 -\frac{|k|}{M} \right)   A_{\varepsilon,L}(k),
\end{eqnarray}
where \[A_{\varepsilon,L}(k):=\int_{X_a} (\psi \circ c^k)\psi \, d\mu_a.\]
Define 
\[\phi:=\hat{\chi} -\vol(\Omega_2) \text{ and } A(k):=\int_{X_a} (\phi \circ c^k) \phi \, d\mu_a.\]
Rogers's second moment formula gives
\begin{eqnarray}
 A(k) &=&  \int_{X_a} \hat{\chi} \cdot (\hat{\chi} \circ c^k) \,\, d\mu_a - \vol(\Omega_2)^2 \nonumber \\ \nonumber &=&  \frac{1}{2} \left(  
\int_{X_a}(\hat{\chi}+\hat{\chi} \circ c^k)^2 \,\, d\mu_a -2\int_{X_a} \hat{\chi}^2 \,\, d\mu_a - 2 \vol(\Omega_2)^2  \right) \\ \nonumber &=& \frac{1}{2} \left( \left(\int_{\mathbb R^l} (\chi + \chi \circ c^k) \,\, \right)^2 + \int_{\mathbb R^l} (\chi + \chi \circ c^k)^2 \,\, -2 \left(\left(\int_{\mathbb R^l} \chi \right)^2 + \int_{\mathbb R^l} \chi^2 \,\, \right) -2 \vol(\Omega_2)^2 \right) \nonumber \\ &=& \int_{\mathbb R^l} \chi \cdot (\chi \circ c^k) \,\, d \mathbf x \nonumber \\ &=& \vol(\Omega_2 \cap c^{-k} \Omega_2).
\end{eqnarray}
Note that $A(k)=0$ for $k\neq 0$ and $A(0)=\vol(\Omega_2).$ Also in view of Lemma \ref{lem:f_epsilon_L_chi_hat_L_2}, we get that
\begin{equation}\label{equ:A_epsilon_L_A_k}
    |A_{\varepsilon,L}(k)-A(k)|\ll L^{-(l-3)/2}+\sqrt{\varepsilon},
\end{equation}
 where the implied constant is independent of $k.$

Now we choose $K=[b\log M]$ (where $b$ is a constant to be chosen later) and divide the variance computation into two parts. Define
\[\|H_{M}^{(\varepsilon,L)}\|_2^2:=V_M^{\text{near}} +V_M^{\text{far}},\]
where \[V_M^{\text{near}}:=\sum_{|k| \leq K }  \left(1 -\frac{|k|}{M} \right)   A_{\varepsilon,L}(k) \]
and \[V_M^{\text{far}}:=\sum_{K<|k| <M }  \left(1 -\frac{|k|}{M} \right)  A_{\varepsilon,L}(k).\]

We compute $V_M^{\text{near}} $ and $V_M^{\text{far}}$ separately.
\begin{eqnarray}
    V_M^{\text{near}} &=& \sum_{|k| \leq K }  \left(1 -\frac{|k|}{M} \right)   A(k) + \sum_{|k| \leq K }  \left(1 -\frac{|k|}{M} \right)   (A_{\varepsilon,L}(k)-A(k)) \nonumber \\ &=& \vol(\Omega_2) + \sum_{|k| \leq K }  \left(1 -\frac{|k|}{M} \right)   (A_{\varepsilon,L}(k)-A(k)), \,\, \text{ since } A(k)=0 \text{ for } k \ne 0 \nonumber \\ &\leq& \vol(\Omega_2) +(2K+1) |A_{\varepsilon,L}(k)-A(k)| \nonumber \\ &\ll& K\left(L^{-(l-3)/2}+\sqrt{\varepsilon} \right), \,\,\text{using } \eqref{equ:A_epsilon_L_A_k}.
\end{eqnarray}
Also using the quantitative mixing estimate Theorem \ref{thm:quan_mixing} (and Lemma \ref{lem:tranc_Siegel_property}), there exists $\delta_0>0$ and $q \in \mathbb N$ such that
\[|A_{\varepsilon,L}(k)|=\left|\int_{X_a} (\psi \circ c^k) \psi \, d\mu_a \right| \ll e^{-\delta_0|k|}S_q(\psi)^2 \ll e^{-\delta_0|k|} \varepsilon^{-2q}L^2.\]
Now we have
\begin{equation}
    V_M^{\text{far}} \leq \sum_{ K<|k| <M } |A_{\varepsilon,L}(k)| \leq \sum_{|k|>K  } |A_{\varepsilon,L}(k)| \ll \varepsilon^{-2q}L^2 \sum_{|k|>K} e^{-\delta_0 |k|} \ll \varepsilon^{-2q}L^2 e^{-\delta_0 K}.
\end{equation}
Therefore, combining both estimates, we get
\begin{equation}
    \left| \|H_M^{(\varepsilon,L)}\|_2^2 - \vol(\Omega_2)\right| \ll K\left(L^{-(l-3)/2}+\sqrt{\varepsilon} \right) + \varepsilon^{-2q}L^2 e^{-\delta_0 K}.
\end{equation}
Recall $K=[b \log M].$ Now choose 
\[\varepsilon = M^{-2} \text{ and } L=M^{1/(l-2)} \log M.\]
Then 
\[K\sqrt{\varepsilon} \rightarrow 0 \quad \text{and}\quad  KL^{-(l-3)/2} \ll M^{-\frac{l-3}{2(l-2)}} (\log M)^{1-\frac{l-3}{2}} \rightarrow 0  \]
as $M \rightarrow \infty$ for $l \geq 4.$

Also,
$\varepsilon^{-2q}L^2 e^{-\delta_0 K} \ll M^{4q+\frac{2}{l-2}-b\delta_0} (\log M)^2 \rightarrow 0$ as $M\rightarrow \infty$, provided we choose
 \[b>\frac{4q+2/(l-2)}{\delta_0}.\]
Hence it follows that 
\[\|H_M^{(\varepsilon,L)}\|_2^2 \rightarrow \vol(\Omega_2) \quad \text{as } M\rightarrow \infty.\]
 \\

\noindent \textbf{Computation of cumulants:} 

First, let us write 
\begin{equation}\label{equ:H_epsilon_L_psi}
       H_M^{(\varepsilon,L)}:=\frac{1}{\sqrt{M}} \sum_{k=0}^{M-1}   \psi_k,
   \end{equation}
where $\psi_k=\hat{f_{\varepsilon}}^{(L)} \circ c^k - \mu_a \left( \hat{f_{\varepsilon}}^{(L)} \right)$ and $\mu_a \left( \hat{f_{\varepsilon}}^{(L)} \right)=\displaystyle \int_{X_a} \hat{f_{\varepsilon}}^{(L)} \,\, d\mu_a.$

\noindent By the multilinearity of $\text{cum}_{[r]}$ we get that
\begin{equation}\label{equ:cum_H_M_epsilon_L}
\text{cum}_{[r]} \left(H_M^{(\varepsilon,L)} \right)=\frac{1}{M^{r/2}} \sum_{k_1,\dots,k_r=0}^{M-1} \text{cum}_{[r]}(\psi_{k_1},\dots,\psi_{k_r})
\end{equation}
Now we want to decompose the sum of \eqref{equ:cum_H_M_epsilon_L} into sub-sums according to our convenience following \cite[Proposition 6.2  ]{BG1'} and  \cite[Eq. (3.8)]{BG1}.

\begin{proposition}\label{prop:decom_R}
    Let $r \in \mathbb N$ with $r \geq 3.$ Let $0 \leq \alpha < \beta$ and a partition $\mathcal Q$ of $\{1,\dots,r\}$ be given, and $\mathbf{k}=(k_1,\dots,k_r) \in \mathbb R^r_+$. We define
    \begin{equation*}
        \Delta(\alpha) := \{\mathbf{k} \in \mathbb R^r_+: |k_i -k_j| \leq \alpha \,\, \forall i,j\}
    \end{equation*}
    and 
    \begin{equation*}
\Delta_{\mathcal Q}(\alpha,\beta)=\left\{\mathbf{k} \in \mathbb R^r_+: \max_{I \in \mathcal Q} \max_{i,j \in I} \{|k_i -k_j|\} \leq \alpha \,\, \text{and} \,\, \min_{I,J \in \mathcal Q, I \neq J} \min_{i \in I,j\in J} \{|k_i -k_j|\}>\beta \right\}.
    \end{equation*}
    Then, given $0=\alpha_0 < \beta_1 < \alpha_1=(3+r)\beta_1 <\beta_2 < \dots < \beta_r < \alpha_{r-1}=(3+r) \beta_{r-1} < \beta_{r}$, we have
\begin{equation*}
    \mathbb R^{r}_+=\Delta(\beta_r) \cup \left(\bigcup_{j=0}^{r-1} \bigcup_{ |\mathcal Q| \geq 2} \Delta_{\mathcal Q} (\alpha_j,\beta_{j+1})\right).
\end{equation*}    
Intersecting with $\{0,1,\dots,M-1\}^r,$ we get that 
\begin{equation*}
    \{0,1,\dots,M-1\}^r=\Omega(\beta_r,M)  \cup \left(\bigcup_{j=0}^{r-1} \bigcup_{ |\mathcal Q| \geq 2} \Omega_{\mathcal Q} (\alpha_j,\beta_{j+1},M)\right),
\end{equation*}
where  
\begin{equation*}
\Omega(\beta_r,M):= \{0,1,\dots,M-1\}^r \cap \Delta(\beta_r),
\end{equation*}
and
\begin{equation*}
    \Omega_{\mathcal Q} (\alpha_j,\beta_{j+1},M):= \{0,1,\dots,M-1\}^r \cap \Delta_{\mathcal Q}(\alpha_j,\beta_{j+1}).
\end{equation*}
\end{proposition}
 To estimate the cumulant \eqref{equ:cum_H_M_epsilon_L}, the strategy involves  separately estimating the sums over  $\Omega(\beta_r,M)$ and 
$\Omega_{\mathcal Q} (\alpha_j,\beta_{j+1},M)$. After accounting for all other factors, the sequences $\{\alpha_j\}$ and $\{\beta_j\}$ are chosen at the final step. \\

\noindent \textbf{Case: 1 Summing over }$\mathbf{\Omega(\beta_r,M)}.$

\noindent Suppose the index $\mathbf{k}=(k_1,\dots,k_r)$ runs over ${\Omega(\beta_r,M)}.$ Then
\begin{equation*}
    |k_i -k_j| \leq \beta_r \,\, \text{for all} \,\, i,j.
\end{equation*}
Given this, we have
\begin{equation*}
     |\Omega(\beta_r,M)| \ll M\beta_r^{r-1}
\end{equation*}
We claim that 
\begin{equation}\label{equ:cumm_case_i}
\frac{1}{M^{r/2}} \sum_{(k_1,\dots,k_r) \in \Omega(\beta_r,M)} |\text{cum}_{[r]}(\psi_{k_1},\dots,\psi_{k_r})| \ll_r M^{1-r/2} \beta_r^{r-1} \left\| \hat{f_{\varepsilon}}^{(L)} \right\|^{(r-l)^+}_{\infty}
\end{equation}
where $(r-l)^+:=\max \{0,r-l\}.$ 

To prove \eqref{equ:cumm_case_i}, it is enough to show that 
\begin{equation}\label{equ:cumm_case_i_redu}
    \int_{X_a} |\psi_{k_1} \dots \psi_{k_r}| \,\, d\mu_a \ll_{r} \left\| \hat{f_{\varepsilon}}^{(L)} \right\|^{(r-l)^+}_{\infty}
\end{equation}
First, consider the case $r < l$. In this case we apply the generalized H$\ddot{\text{o}}$lder inequality to the $l-1$ functions $\psi_{k_1},\dots,\psi_{k_r},1,\dots,1$ to get
\begin{equation}\label{equ:hold}
\int_{X_a} |\psi_{k_1} \dots \psi_{k_r}| \,\, d\mu_a \ll \|\psi_{k_1}\|_{(l-1)} \dots \|\psi_{k_r}\|_{(l-1)}, 
\end{equation}
{where $\|\cdot\|_{(l-1)}$ denotes the $L^{(l-1)}$ norm.}

\noindent Now for any $k$
\begin{eqnarray*}
\|\psi_k\|_{(l-1)} &\leq &\left\|\hat{f_{\varepsilon}}^{(L)}\right\| _{(l-1)} + \mu_a \left(\hat{f_{\varepsilon}}^{(L)} \right), \,\, \text{by using $G_{a}$-invariance of $\mu_a$}  \\ & \leq & \|\alpha_{l,a}\|_{(l-1)} + \|f_{\varepsilon}\|_1, \,\, \text{by Proposition \ref{prop:Siegel_alpha_bound_affine}}  \\ &\leq &  \|\alpha_{l,a}\|_{(l-1)} + \|f_{\varepsilon} - \chi\|_1 + \|\chi\|_1 \\ & \ll& \|\alpha_{l,a}\|_{(l-1)} + \varepsilon + \vol(\Omega_2). 
\end{eqnarray*}
Combining the above estimate with  \eqref{equ:hold}, we get 

\begin{equation}\label{equ:psi_k_1_r}
    \int_{X_a} |\psi_{k_1} \dots \psi_{k_r}| \,\, d\mu_a \ll 1.
\end{equation}
This shows that \eqref{equ:cumm_case_i_redu} holds if $r <l.$ Next we consider the case when $r\geq l.$ In this case
\begin{eqnarray}
     \int_{X_a} |\psi_{k_1} \dots \psi_{k_{l-1}} \psi_{k_l}\psi_{k_{l+1}} \dots \psi_{k_r}| \,\, d\mu_a  & \leq & \|\psi_{k_{l}} \dots \psi_{k_r}\|_{\infty} \int_{X_a} |\psi_{k_1} \dots \psi_{k_{l-1}}| \,\, d\mu_a \nonumber \\ &\leq& 2^{r-l} \left\|\hat{f_{\varepsilon}}^{(L)}\right\|_{\infty}^{r-l} \|\psi_{k_1}\|_{(l-1)} \dots \|\psi_{k_r}\|_{(l-1)} \nonumber \\ &\ll&  2^{r-l} \left\|\hat{f_{\varepsilon}}^{(L)}\right\|_{\infty}^{r-l}, \,\, \text{using above estimate for $\|\psi_k\|_{(l-1)}$} \nonumber \\ &\ll& \left\|\hat{f_{\varepsilon}}^{(L)}\right\|_{\infty}^{r-l} \nonumber.
\end{eqnarray}
Therefore \eqref{equ:cumm_case_i_redu} holds and hence we get our desired result.\\

\noindent \textbf{Case: 2 Summing over }$\mathbf{\Omega_{\mathcal Q}(\alpha_j,\beta_{j+1},M)}$ \textbf{with} $\mathbf{  |\mathcal Q| \geq 2}.$

\noindent Suppose that $\mathcal Q=\{J_1,\dots,J_d\}$ with $d \geq 2.$ Given an arbitrary subset $I$ of $\{1,\dots,r\}$ first we want to show that 

\begin{equation}\label{equ:coupling}
    \int_{X_a} \prod_{i \in I} \psi_{k_i} \,\, d\mu_a 
 \approx \prod_{t=0}^{d} \int_{X_a} \prod_{i \in I \cap J_t} \psi_{k_i} \,\,d\mu_a.
\end{equation}
If  $I \subseteq J_t$ $(\text{hence} \,\, I \cap J_t \,\, \text{is empty for all indices} \,\, t \,\, \text{except one})$ for some $t=1,\dots,d$, then it is easy to observe that equality holds in \eqref{equ:coupling}. In other cases, we will show that equality in \eqref{equ:coupling} holds with an error term. 

\noindent If $I=I_1 \sqcup \dots \sqcup I_s$, where each $I_t$ is nonempty and $I_t \in \{I \cap J:J \in \mathcal Q\}$ for $t=1,\dots,s.$ Let $g=\hat{f_{\varepsilon}}^{(L)} - \mu_a \left( \hat{f_{\varepsilon}}^{(L)} \right)$, $\mathbf{f}=\hat{f_{\varepsilon}}^{(L)}$ and $k_{I_j}:=\max \{k_i:i \in I_j \}.$ Then

\begin{eqnarray} 
\int_{X_a} \prod_{i \in I} \psi_{k_i} \,\, d \mu_a \nonumber &=& \int_{X_a} \prod_{i \in I} g \circ c^{k_i} \,\, d \mu_a \nonumber \\ &=& \int_{X_a} \prod_{j=1}^{s} \left( \prod_{i \in I_j}g \circ c^{k_i} \right) \,\, d \mu_a \nonumber \\  &=& \int_{X_a} \prod_{j=1}^{s} \left( \prod_{i \in I_j}g \circ c^{k_i-k_{I_j}} \right) \circ c^{k_{I_j}} \,\, d \mu_a \nonumber \\ &=& \int_{X_a} \prod_{j=1}^{s} \left( \sum_{K_j \subseteq I_j } (-\mu_{a}\left( \mathbf{f} \right))^{ |K_j|}   \prod_{i \in I_j \setminus K_j}   \mathbf{f} \circ c^{k_i-k_{I_j}}  \right) \circ c^{k_{I_j}} \,\, d \mu_a \nonumber \\ &=& \int_{X_a}    \sum_{K_j \subseteq I_j, j=1,\dots,s } (-\mu_{a}\left( \mathbf{f} \right))^{\sum_{d=1}^{s} |K_d|} \prod_{j=1}^{s} \left( \prod_{i \in I_j \setminus K_j}   \mathbf{f} \circ c^{k_i-k_{I_j}}  \right) \circ c^{k_{I_j}} \,\, d \mu_a  \nonumber \\ &=&      \sum_{K_j \subseteq I_j,j=1,\dots,s } (-\mu_{a}\left( \mathbf{f} \right))^{\sum_{d=1}^{s} |K_d|} \int_{X_a} \prod_{j=1}^{s} \left( \prod_{i \in I_j \setminus K_j}   \mathbf{f} \circ c^{k_i-k_{I_j}}  \right) \circ c^{k_{I_j}} \,\, d \mu_a  \nonumber 
\end{eqnarray}
Let $\mathbf{f}_{K_j}:=\prod_{i \in I_j \setminus K_j}   \mathbf{f} \circ c^{k_i-k_{I_j}} $ for $j=1,\dots,s.$ Then from above we have
\begin{equation}\label{equ:coupling_redu_i}
\int_{X_a} \prod_{i \in I} \psi_{k_i} \,\, d \mu_a =\sum_{K_j \subseteq I_j,j=1,\dots,s } (-\mu_{a}\left( \mathbf{f} \right))^{\sum_{d=1}^{s} |K_d|} \int_{X_a} \prod_{j=1}^{s} \mathbf{f}_{K_j} \circ c^{k_{I_j}} \,\, d \mu_a.
\end{equation}

We want to use the quantitative estimate for higher order correlations, i.e., Theorem \ref{thm:quan_mixing} to compute the above integral. Before that, let us compute a few other things. Using Lemma \ref{lem:pro_sob_norm}, we get
\begin{eqnarray}
    \prod_{j=1}^{s}S_q(\mathbf{f}_{K_j}) &=& \prod_{j=1}^{s}S_q \left(\prod_{i \in I_j \setminus K_j}   \mathbf{f} \circ c^{k_i-k_{I_j}} \right) \nonumber \\ &\ll& \prod_{j=1}^{s} \left(\prod_{i \in I_j \setminus K_j}   S_q \left(\mathbf{f} \circ c^{k_i-k_{I_j}} \right) \right) \nonumber \\ & \ll & \prod_{j=1}^{s} S_q(\mathbf{f})^{|I_j \setminus K_j|} \prod_{i \in I_j \setminus K_j} \left\|(c^{-1})^{k_{I_j} -k_i} \right\|_{\text{op}}^{q} \nonumber \\&\ll& \prod_{j=1}^{s} S_q(\mathbf{f})^{|I_j \setminus K_j|}   \left\|c^{-1} \right\|_{\text{op}}^{q |I_j \setminus K_j|\alpha_j},\,\, \text{since} \,\, k_{I_j}-k_i \leq \alpha_j \,\, \forall i \in I_j \setminus K_j\nonumber \\ & \ll & S_q(\mathbf{f})^r \|c^{-1}\|^{qr \alpha_j}_{\text{op}}. \nonumber
\end{eqnarray}
Hence \begin{equation}\label{equ:S_q_f_K_j}
    \prod_{j=1}^{s}S_q(\mathbf{f}_{K_j}) \ll S_q(\mathbf{f})^r e^{qr \alpha_j \tau}, \,\, \text{where} \,\, \tau=\log \|c^{-1}\|_{\text{op}}>0.
\end{equation}

 \noindent Using \eqref{equ:coupling_redu_i}, \eqref{equ:S_q_f_K_j} and Theorem \ref{thm:quan_mixing}, we get
\begin{equation}\label{equ:psi_f_K_j}
\int_{X_a} \prod_{i \in I} \psi_{k_i} \,\, d \mu_a =   \sum_{K_j \subseteq I_j,j=1,\dots,s } (-\mu_{a}\left( \mathbf{f} \right))^{\sum_{d=1}^{s} |K_d|} \left(\prod_{j=1}^{s} \int_{X_a} \mathbf{f}_{K_j} \,\, d\mu_a 
     + O_{q,r}(E')\right)
\end{equation}
where

\begin{equation}\label{equ:E_e_S_q}
    E' = e^{-\delta \min_{i\neq j} \rho_{H_a}\left(c^{k_{I_i}},c^{k_{I_j}}\right)} \prod_{j=1}^{s} S_q(\mathbf{f}_{K_j})  \ll  e^{-\delta \min_{i\neq j} \rho_{H_a}\left(1,c^{k_{I_j}-k_{I_i}}\right)}  S_q(\mathbf{f})^r e^{qr \alpha_j \tau}  
\end{equation}
First note that for any $i \neq j,$ $|k_{I_j}-k_{I_i}|>\beta_{j+1}.$ Suppose that $k_{I_j}-k_{I_i} >0.$ Then using Lemma \ref{lem:operator_norm}, we get $0<D_1 \leq 1, D_2>0  $ such that
\begin{equation}\label{equ:rho_G_a}
 \rho_{H_a}\left(1,c^{k_{I_j}-k_{I_i}}\right) \geq D_1 \log \|c^{k_{I_j}-k_{I_i}}\|_{\text{op}} - D_2 \geq D_1 \beta_{j+1} \log \lambda - D_2.
\end{equation}
Let $\delta'=\delta D_1 \log \lambda$. Combining \eqref{equ:E_e_S_q} and \eqref{equ:rho_G_a} we get that
\begin{equation}
    E' \ll e^{-\delta' \beta_{j+1}} S_q(\mathbf{f})^r e^{qr \alpha_j \tau} = e^{-(\delta' \beta_{j+1}-qr \alpha_j \tau)} S_q(\mathbf{f})^r
\end{equation}
Now if we let $E=e^{-(\delta' \beta_{j+1}-qr \alpha_j \tau)} S_q(\mathbf{f})^r ,$ then from \eqref{equ:psi_f_K_j} we get 
 \begin{eqnarray}
\int_{X_a} \prod_{i \in I} \psi_{k_i} \,\, d \mu_a &=&   \sum_{K_j \subseteq I_j,j=1,\dots,s } (-\mu_{a}\left( \mathbf{f} \right))^{\sum_{d=1}^{s} |K_d|} \prod_{j=1}^{s} \int_{X_a} \mathbf{f}_{K_j} \,\, d\mu_a 
     + O_{q,r}(E) \nonumber \\ &=& \prod_{j=1}^{s} \sum_{K_j \subseteq I_j } (-\mu_{a}\left( \mathbf{f} \right))^{ |K_j|} \int_{X_a} \mathbf{f}_{K_j} \,\, d\mu_a 
     + O_{q,r}(E) \nonumber\\ &=& \prod_{j=1}^{s} \int_{X_a} \prod_{i \in I_j} \left(g \circ c^{k_i} \right) \,\, d\mu_{a} + O_{q,r}(E) \nonumber \\ &=& \prod_{j=1}^{s} \int_{X_a} \prod_{i \in I_j} \psi_{k_i}\,\, d\mu_{a} + O_{q,r}(E) \nonumber
\end{eqnarray}
Hence finally for any $\mathbf{k}=(k_1,\dots,k_r) \in \Omega_{\mathcal Q}(\alpha_j,\beta_{j+1},M)$ with $ |\mathcal Q| \geq 2,$ we have

\begin{equation}\label{equ:coupling_final}
    \int_{X_a} \prod_{i \in I} \psi_{k_i} \,\, d\mu_a 
 = \prod_{t=0}^{d} \int_{X_a} \prod_{i \in I \cap J_t} \psi_{k_i} \,\,d\mu_a + O_{q,r}(E),
\end{equation}
where $E=e^{-(\delta' \beta_{j+1}-qr \alpha_j \tau)} S_q(\mathbf{f})^r .$ Now summing the above estimate over all partitions $\mathcal P$ of $\{1,\dots,r\}$ and denoting an element of $\mathcal P$ by $I,$ we get    
\begin{eqnarray}
\text{cum}_{[r]}(\psi_{k_1},\dots,\psi_{k_r})= \text{cum}_{[r]}(\psi_{k_1},\dots,\psi_{k_r} | \mathcal Q) + O_{q,r}(E)
\end{eqnarray}
Now by using Proposition \ref{prop:condi_cumu}, we get 
\begin{equation}
|\text{cum}_{[r]}(\psi_{k_1},\dots,\psi_{k_r})| \ll_{q,r} E
\end{equation}

\noindent \textbf{Final estimate on the cumulants:}

Combining all the estimates, we get 
 
 \begin{equation}\label{equ:total_cummu}
     \text{cum}_{[r]} \left(H_M^{(\varepsilon,L)} \right) \ll M^{1-r/2} \beta_r^{r-1} L^{(r-l)^+} + \varepsilon^{-qr}L^r M^{r/2}  \max_j \left\{ e^{-(\delta' \beta_{j+1}-qr \alpha_j \tau)} \right\}
 \end{equation}
Now we choose the parameters $\alpha_j$ and $\beta_j$ such that the right hand side of \eqref{equ:total_cummu} goes to zero as $M \rightarrow \infty.$ We do this by choosing a single parameter $\rho>0.$ We define the parameters inductively by
\begin{equation}\label{equ:beta_j}
    \beta_1=\rho \quad \text{and} \quad \beta_{j+1}=\max \{\rho+(3+r)\beta_j,\rho+(\delta')^{-1} r(3+r)q\tau\}
\end{equation}
for $j=1,\dots,r-1.$

\noindent The above choice of $\beta_{j+1}$ fulfills the requirement, i.e., $\alpha_j=(3+r)\beta_j < \beta_{j+1}$ of Proposition  \ref{prop:decom_R}. Also we have
\begin{equation}
\delta' \beta_{j+1}-qr\tau \alpha_j \geq \delta' \rho >0.
\end{equation}
By induction, it easily follows that $\beta_r \ll_r \rho.$ Hence from \eqref{equ:total_cummu} we get
\begin{equation}\label{equ:total_cummu_1}
\text{cum}_{[r]} \left(H_M^{(\varepsilon,L)} \right) \ll M^{1-r/2} {\rho}^{r-1} L^{(r-l)^+} + \varepsilon^{-qr}L^r M^{r/2}   e^{-\delta' \rho}.
\end{equation}
Further we want to choose parameters $\eta,$ $\varepsilon,$ and $L$ such that 
\begin{equation}\label{equ:cummu_part_1}
 M^{1-r/2} {\rho}^{r-1} L^{(r-l)^+} \rightarrow 0  
\end{equation}
and
\begin{equation}\label{equ:cummu_part_2}
     \varepsilon^{-qr}L^r M^{r/2}   e^{-\delta' \rho} \rightarrow 0
\end{equation}
   as $M \rightarrow \infty.$ 
   
   \noindent {We choose the parameters as   
   \begin{equation}\label{equ:final_para}
       \rho=C_r \log M,\quad  \varepsilon = M^{-2} \quad \text{and}\quad L=M^{1/(l-2)} \log M,
   \end{equation}
    where $C_r>0$ is some constant to be chosen later. It is worth noting that the above chice is consistent with our earlier choice \eqref{equ:epsilon_L} in Proposition \ref{prop:H_epsilon_L}.}
   
{    First consider the case when $3 \leq r \leq l$. Then $(r-l)^+=0$ and hence $L^{(r-l)^+}=1.$ Therefore, in this case \eqref{equ:cummu_part_1}
follows by our choice \eqref{equ:final_para} of the parameter $\rho $ as above, since $r \geq 3$, $1-r/2 \leq -1/2$ and $\rho^{r-1}=o(M^{1/2}).$  }  

{\noindent Now assume that $r>l.$ In this case our choice \eqref{equ:final_para} of the parameter $L$ gives
\[ M^{1-r/2} \rho^{r-1}L^{r-l} \ll M^{\frac{r(4-l)-4}{2(l-2)}}(\log M)^{2r-l-1} \rightarrow 0\]
as $M \rightarrow \infty,$ provided $l\geq 4.$}

{Also we have
\begin{equation}
    \varepsilon^{-qr}L^r M^{r/2}   e^{-\delta' \rho} = M^{2qr+\frac{r}{l-2}+\frac{r}{2}-\delta' C_r} (\log M)^r \rightarrow 0
\end{equation}
 as $M\rightarrow \infty,$ provided
 \[C_r > \frac{2qr+\frac{r}{l-2}+\frac{r}{2}+1}{\delta'}.\]  }

  {Hence for such a choice of $C_r$, $\rho,$ $\varepsilon$ and $L$, our desired conditions  \eqref{equ:cummu_part_1} and \eqref{equ:cummu_part_2} holds. This shows that
  \[\lim_{M \rightarrow \infty} \text{Cum}_{[r]} \left(H_M^{(\varepsilon,L)} \right)=0 \quad \text{for all} \,\, r \geq 3.\] Finally an immediate application of Theorem \ref{thm:CLT_crit} proves Theorem \ref{thm:H_epsilon_L_CLT}.  }

\section{Proof of the CLT for congruence unimodular lattices}
Throughout this section, we will denote the set of all congruence unimodular lattices on $\mathbb R^l$ by $X_c,$ where $l=m+n.$ Let $g : \mathbb R^l \rightarrow \mathbb R$ be a Borel measurable function with compact support. The Siegel transform of $g$ is $\hat{g}:X_{c} \rightarrow \mathbb R$ defined by
\begin{equation}\label{equ:31}
\hat{g}(\Lambda) =\sum_{\lambda \in \Lambda} g(\lambda) \quad \text{for} \,\, \Lambda \in X_c.
\end{equation}
Now let us record the analogs of Siegel's transform and Rogers's second moment formula for congruence lattice from \cite{GKY}. The proof of the same can also be found in \cite{AGH}.
\begin{proposition}\cite[Theorem 3.2]{GKY}.\label{prop:Siegel_cong}
    Let $g : \mathbb R^l \rightarrow \mathbb R$ be a bounded Borel measurable function with compact support. Then
    \begin{equation}\label{equ:Siegel_cong}
 \int_{X_c} \hat{g}(\Lambda) \,\, d\mu_c(\Lambda)=\int_{\mathbb R^l} g(\mathbf x) \,\, d\mathbf x
    \end{equation}
    \begin{equation}\label{equ:Roger_cong}
    \int_{X_c} {\hat{g}(\Lambda)}^2 \,\, d\mu_c(\Lambda)= \left(\int_{\mathbb R^l} g(\mathbf x) \,\, d\mathbf x \right)^2 +\frac{1}{\zeta_N(m+n)} \sum_{\tiny \begin{array}{c} s_1 \geq 1\\ \gcd(s_1,N)=1 \end{array} } \sum_{ \tiny \begin{array}{c} s_2 \in \mathbb Z \setminus \{0\} \\ s_2 \equiv s_1  (\mod  N) \end{array}} \int_{\mathbb R^l} {g(s_1\mathbf x)} g(s_2 \mathbf x) \,\, d\mathbf x,
    \end{equation}
    where $\zeta_N$ is defined in Theorem \ref{thm:main_congruence}.
\end{proposition}
{First recall that 
\begin{equation}\label{equ:25'}
    {c}_0=\diag (2^{u_1},\dots,2^{u_m},2^{-1},\dots,2^{-1}) \in \SL_l(\mathbb R).
\end{equation}}
By arguments similar to the previous section, it is easy to see that Theorem \ref{thm:main_congruence} is equivalent to the following theorem.
\begin{theorem}\label{thm:cong_reduc}
Under the hypothesis of Theorem \ref{thm:main_congruence},
       
    \begin{equation}\label{equ:R_M}
        R_M:=\frac{1}{\sqrt{M}} \sum_{k=0}^{M-1} \left(\hat{\chi} \circ c_0^k - M \cdot \vol(\Omega_2)\right) \rightarrow N(0,\eta^2\sigma_c^2) \,\, \text{in distribution as} \,\, M \rightarrow \infty,
    \end{equation}
    where $\sigma_c^2$ is given by Theorem \ref{thm:main_congruence}, $\eta=\vol(\Omega_2)^{1/2}$ and $\hat{\chi}$ is the Siegel transform of $\chi_{\Omega_2}$ on $X_c.$
\end{theorem}
Again, as we have seen in the case of affine unimodular lattices, here also Theorem \ref{thm:cong_reduc} is equivalent to proving the following theorem.
\begin{theorem}\label{thm:R_epsilon_L}
{Suppose that the parameters $\varepsilon$ and $L$ satisfy
    \begin{equation}\label{equ:para_congruence}
        \varepsilon = M^{-2} \quad \text{and}\quad L=M^{1/(l-2)} \log M.
    \end{equation} }
Under the hypothesis of Theorem \ref{thm:main_congruence},
    \begin{equation}\label{equ:R_epsilon_L}
       R_M^{(\varepsilon,L)}:=\frac{1}{\sqrt{M}} \left( \sum_{k=0}^{M-1}   \hat{g_{\varepsilon}}^{(L)} \circ c_0^k - M \int_{X_c} \hat{g_{\varepsilon}}^{(L)} \,\,  d\mu_c \right) \rightarrow N(0,\eta^2\sigma_c^2) \,\, \text{in distribution as} \,\, M \rightarrow \infty ,
   \end{equation}
   where $\hat{g_{\varepsilon}}^{(L)}=\hat{g_{\varepsilon}} \eta_L \in C_{c}^{\infty}(X_c)$ are smooth approximations to $\hat{\chi}$, $\sigma_c^2$ is given by Theorem \ref{thm:main_congruence} and $\eta=\vol(\Omega_2)^{1/2}$.

\end{theorem}
Again, we use the {CLT criteria of Theorem \ref{thm:CLT_crit} to prove} Theorem \ref{thm:R_epsilon_L}. Proving Theorem \ref{thm:R_epsilon_L} amounts to show that for \eqref{equ:para_congruence} as the choice of parameters $\varepsilon$ and $L$, we have  

\begin{equation}\label{equ:cong_CLT_condition_var}
\sigma^2 :=\lim_{M \rightarrow \infty} \|R_M^{(\varepsilon,L)}\|_2^2 < \infty
\end{equation}

\begin{equation}\label{equ:cong_CLT_condition_cum}
\lim_{M \rightarrow \infty} \text{Cum}_{[r]} \left(R_M^{(\varepsilon,L)} \right)=0 \quad \text{for all} \,\, r \geq 3.
\end{equation}

The computation of cumulant (i.e., Equation \eqref{equ:cong_CLT_condition_cum}) is analogous to the affine case (with appropriate modifications) discussed in detail in the previous section. Hence, we omit the details here. Rather, we compute the variance for the congruence case. Before that, we need the following lemma:
\begin{lemma}\label{lem:main_est_cong}\begin{enumerate}[label=({{\roman*}})]
    \item {$\|\hat{g_{\varepsilon}}^{(L)} -\hat{\chi}\|_2 \ll L^{-(l-3)/2}+\sqrt{\varepsilon}.$ }
    \item \begin{equation}
       { \left|\int_{X_c} (\Psi \circ c_0^k)\Psi \, d\mu_c - \int_{X_c} (\Phi \circ c_0^k)\Phi \, d\mu_c\right| \ll L^{-(l-3)/2}+\sqrt{\varepsilon}, }
    \end{equation}
  {  where \begin{equation}
      \Psi=\hat{g_{\varepsilon}}^{(L)} -\int_{X_c} \hat{g_{\varepsilon}}^{(L)} \,\,  d\mu_c \quad \text{and}\quad  \Phi:=\hat{\chi} -\vol(\Omega_2)
    \end{equation} }
\end{enumerate}
\begin{proof}
\begin{enumerate}[label=({{\roman*}})]
\item \begin{equation}\label{equ:est1}
    \left\| \hat{g_{\varepsilon}}^{(L)} -\hat{\chi} \right\|_2 \ll      \left\| \hat{g_{\varepsilon}} - \hat{g_{\varepsilon}}^{(L)} \right\|_2 + 
  \left\|\hat{g_{\varepsilon}} - \hat{\chi}\right\|_2
\end{equation}
Using congruence analog of Lemma \ref{lem:tranc_Siegel_property} (and Remark \ref{rem:trun_sei_con}), we get
  \begin{equation}\label{equ:g_L_g_esp_est}
      \left\| \hat{g_{\varepsilon}} - \hat{g_{\varepsilon}}^{(L)} \right\|_2 \ll L^{-(l-3)/2}.
  \end{equation}
  Now first recall that the smoothing functions $\{g_{\varepsilon}\}$ are chosen so that 
  \[\|g_{\epsilon}-\chi\|_1 \ll \varepsilon \quad \text{and} \quad \|g_{\epsilon}-\chi\|_2 \ll \sqrt{\varepsilon}. \]
  Using the Rogers's formula Proposition \ref{prop:Siegel_cong} and the Cauchy-Schwarz inequality, we get 
  \begin{eqnarray}\label{equ:est2}
      \|\hat{g_{\varepsilon}} - \hat{\chi}\|_2^2 &\leq& \|g_{\epsilon}-\chi\|_1^2 \nonumber \\ &+& \frac{1}{\zeta_N(l)} \sum_{\tiny \begin{array}{c} s_1 \geq 1\\ \gcd(s_1,N)=1 \end{array} } \sum_{ \tiny \begin{array}{c} s_2 \in \mathbb Z \setminus \{0\} \\ s_2 \equiv s_1  (\mod  N) \end{array}} \int_{\mathbb R^l} |{(g_{\varepsilon}-\chi)(s_1\mathbf x)} (g_{\varepsilon}-\chi)(s_2 \mathbf x)| \,\, d\mathbf x \nonumber \\ & \leq & \|g_{\epsilon}-\chi\|_1^2 + \frac{\|g_{\varepsilon}-\chi\|_2^2}{\zeta_N(l)} \sum_{\tiny \begin{array}{c} s_1 \geq 1\\ \gcd(s_1,N)=1 \end{array} } \sum_{ \tiny \begin{array}{c} s_2 \in \mathbb Z \setminus \{0\} \\ s_2 \equiv s_1  (\mod  N) \end{array}} (s_1 |s_2|)^{-l/2} \nonumber \\ &\leq& \|g_{\epsilon}-\chi\|_1^2 + \frac{\|g_{\varepsilon}-\chi\|_2^2}{\zeta_N(l)} \,\, 2 \left(\sum_{s_1 \geq 1} s_1^{-l/2} \right)\left(\sum_{s_2 \geq 1} s_2^{-l/2}\right) \nonumber \\ &\leq & \|g_{\epsilon}-\chi\|_1^2 +\frac{2\zeta(l/2)^2}{\zeta_N(l)} \|g_{\varepsilon}-\chi\|_2^2  \nonumber \\ &\ll& \varepsilon 
  \end{eqnarray}
 The above estimate combined with  \eqref{equ:est1} and \eqref{equ:g_L_g_esp_est} yields our desired estimate.
  \item The proof of this part is exactly the same as part $(ii)$ of Lemma \ref{lem:f_epsilon_L_chi_hat_L_2}. Hence, we omit the details.
 \end{enumerate}
 \end{proof}
\end{lemma}

{\noindent \textbf{Computation of variance:} }

\noindent Let \[\Psi=\hat{g_{\varepsilon}}^{(L)} -\int_{X_c} \hat{g_{\varepsilon}}^{(L)} \,\,  d\mu_c.\]
Then 
\[R_M^{(\varepsilon,L)}=\frac{1}{\sqrt{M}}\sum_{k=0}^{M-1}\Psi \circ c_0^k.\]
Therefore,
\begin{eqnarray} 
    \|R_M^{(\varepsilon,L)}\|_2^2 &=&\frac{1}{M} \sum_{k_1,k_2=0}^{M-1} \int_{X_c} (\Psi \circ c_0^{k_1}) (\Psi \circ c_0^{k_2}) \,\, d\mu_c \nonumber \\ &=& \frac{1}{M} \sum_{k_1,k_2=0}^{M-1} \int_{X_c} (\Psi \circ c_0^{k_1-k_2}) \Psi \,\, d\mu_c \nonumber \\ & = &  \frac{1}{M} \sum_{\pm k=0}^{M-1} (M -|k|) \int_{X_c} (\Psi \circ c_0^{k}) \Psi \,\, d\mu_c \nonumber \\ &=&   \sum_{|k| <M }  \left(1 -\frac{|k|}{M} \right)   B_{\varepsilon,L}(k),
\end{eqnarray}
where \[B_{\varepsilon,L}(k):=\int_{X_c} (\Psi \circ c_0^k)\Psi \, d\mu_c.\]
Define 
\[\Phi:=\hat{\chi} -\vol(\Omega_2) \text{ and } B(k):=\int_{X_c} (\Phi \circ c_0^k) \Phi \, d\mu_c.\]
Rogers's second moment formula Proposition \ref{prop:Siegel_cong} gives
\begin{eqnarray}
 B(k)&=&  \int_{X_c} \hat{\chi} \cdot (\hat{\chi} \circ c_0^k) \,\, d\mu_c - \vol(\Omega_2)^2 \nonumber \\ \nonumber &=&  \frac{1}{2} \left(  
\int_{X_c}(\hat{\chi}+\hat{\chi} \circ c_0^k)^2 \,\, d\mu_c -2\int_{X_c} \hat{\chi}^2 \,\, d\mu_c - 2 \vol(\Omega_2)^2  \right) \\ \nonumber &=& \frac{1}{2} \left( \left(\int_{\mathbb R^l} (\chi +\chi \circ c_0^k) \,\, d \mathbf x \right)^2 -2\left(\int_{\mathbb R^l} \chi \,\, d \mathbf x \right)^2 -2 \vol(\Omega_2)^2 \right) \\ \nonumber &+& \frac{1}{2 \zeta_N(m+n)}  \sum_{\tiny \begin{array}{c} s_1 \geq 1\\ \gcd(s_1,N )=1 \end{array} } \sum_{ \tiny \begin{array}{c} s_2 \in \mathbb Z \setminus \{0\} \\ s_2 \equiv s_1  (\mod  N) \end{array}}\int_{\mathbb R^l} ( (\chi +\chi \circ c_0^k)(s_1 \mathbf x) (\chi +\chi \circ c_0^k)(s_2 \mathbf x)    \\ \nonumber &-& 2 \chi (s_1 \mathbf x) \chi (s_2 \mathbf x)) \,\, d\mathbf x  \\ \nonumber &=& \frac{1}{ \zeta_N(m+n)}  \left(\sum_{\tiny \begin{array}{c} s_1 \geq 1\\ \gcd(s_1,N)=1 \end{array} } \sum_{ \tiny \begin{array}{c} s_2 \in \mathbb Z \setminus \{0\} \\ s_2 \equiv s_1  (\mod  N) \end{array}}\int_{\mathbb R^l}  (\chi \circ c_0^k)(s_1 \mathbf x) \chi (s_2 \mathbf x)  \,\, d\mathbf x \right).
\end{eqnarray}
Using Tonelli's theorem, we get 
\begin{eqnarray}
    \sum_{k=-\infty}^{\infty} B(k)  &=& \frac{1}{ \zeta_N(m+n)} \sum_{k=-\infty}^{\infty}   \sum_{\tiny \begin{array}{c} s_1 \geq 1\\ \gcd(s_1,N)=1 \end{array} } \sum_{ \tiny \begin{array}{c} s_2 \in \mathbb Z \setminus \{0\} \\ s_2 \equiv s_1  (\mod  N) \end{array}}\int_{\mathbb R^l}  (\chi \circ c_0^k)(s_1 \mathbf x) \chi (s_2 \mathbf x) \,\, d\mathbf x  \\ \nonumber & =& \frac{1}{ \zeta_N(m+n)} \sum_{\tiny \begin{array}{c} s_1 \geq 1\\ \gcd(s_1,N)=1 \end{array} } \sum_{ \tiny \begin{array}{c}  s_2 \in \mathbb Z \setminus \{0\}  \\ s_2 \equiv s_1  (\mod  N) \end{array}}\int_{\mathbb R^l} \left(  \sum_{k=-\infty}^{\infty}   \chi \circ c_0^k\right)(s_1 \mathbf x) \chi (s_2 \mathbf x)  \,\, d\mathbf x.  \\ \nonumber 
\end{eqnarray}
 Now let 
\begin{equation}
   \Omega=  \left\{(\mathbf x,\mathbf y) \in \mathbb R^{m} \times \mathbb R^n: \,\,  |x_i|\|\mathbf y\|^{u_i} < c_i \,\, \text{for}\,\, i=1,\dots,m \right \}.
 \end{equation}
Then, by changing the coordinates to spherical coordinates, we get
 \begin{eqnarray} 
    \int_{\mathbb R^l} \left(  \sum_{k=-\infty}^{\infty}   \chi \circ c_0^k\right)(s_1 \mathbf x) \chi (s_2 \mathbf x)  \,\, d\mathbf x &=&  \int_{\mathbb R^l} \chi_{\Omega}(s_1 \mathbf x) \chi (s_2 \mathbf x)  \,\, d\mathbf x \\ \nonumber &=& \int_{\mathbb R^l} \chi_{\Omega}(s_1 \mathbf x) \chi (|s_2| \mathbf x)  \,\, d\mathbf x \\ \nonumber &=& \vol \left(s_1^{-1} \Omega \cap |s_2|^{-1} \Omega_2 \right) \\ \nonumber &=&  
   \int_{1/|s_2| \leq \|\mathbf y\| \leq 2/s_2} \prod_{i=1}^{m}\left(\frac{2 c_i}{\max \{s_1,|s_2|\}^{1+u_i} \|\mathbf y\|^{u_i}}  \right) d\mathbf y \nonumber \\ &=&  \frac{2^m c_1 \dots c_m \omega_n}{\max \{s_1,|s_2|\}^{m+n} } \int_{1/s_2}^{2/|s_2|}  \frac{1}{r^n} r^{n-1}  dr  \nonumber  \\ &=& \vol(\Omega_2) \max \{s_1,|s_2|\}^{-l}
 \end{eqnarray}
 Hence
 \begin{equation}\label{equ:xi_Omega_vol}
     \sum_{k=-\infty}^{\infty} B(k) =\frac{\vol(\Omega_2)}{\zeta_N(l)}  \sum_{\tiny \begin{array}{c} s_1 \geq 1\\ \gcd(s_1,N)=1 \end{array} } \sum_{ \tiny \begin{array}{c} s_2 \in \mathbb Z \setminus \{0\}  \\ s_2 \equiv s_1  (\mod  N) \end{array}} \max \{s_1,|s_2|\}^{-l}.
 \end{equation}
It is easy to see that
\begin{eqnarray}
     \sum_{\tiny \begin{array}{c} s_1 \geq 1\\ \gcd(s_1,N)=1 \end{array} } \sum_{ \tiny \begin{array}{c} s_2 \in \mathbb Z \setminus \{0\}  \\ s_2 \equiv s_1  (\mod  N) \end{array}} \max \{s_1,|s_2|\}^{-l} &\leq & 2\sum_{s_1,s_2 \geq 1} \max \{s_1,s_2\}^{-l} \nonumber \\ &=&2\sum_{s=1}^{\infty} (2s-1)s^{-l} \nonumber \\ &=& 4\zeta(l-1)-2\zeta(l) \nonumber \\ &<&\infty.
\end{eqnarray}
Therefore $\sigma_c^2=\frac{1}{\vol(\Omega_2)}\sum_{k=-\infty}^{\infty} B(k) <\infty.$ 

Now we show that $\|R_{M}^{(\varepsilon,L)}\|_2^2 \rightarrow \vol(\Omega_2)\sigma_c^2=\sum_{k=-\infty}^{\infty}B(k)$ as $M \rightarrow \infty.$ 

 \noindent First we choose $K=[b\log M]$ (where $b$ is a constant to be chosen later) and divide the variance computation into four parts.
 \begin{eqnarray}
     \|R_{M}^{(\varepsilon,L)}\|_2^2-\sum_{k=-\infty}^{\infty}B(k) &=&\sum_{|k| \leq K }  \left(1 -\frac{|k|}{M} \right)   (B_{\varepsilon,L}(k)-B(k))+\sum_{K<|k|< M }  \left(1 -\frac{|k|}{M} \right)   B_{\varepsilon,L}(k) \nonumber \\ &-& \sum_{|k|\leq K} \frac{|k|}{M} B(k)-\sum_{|k|>K} B(k).
 \end{eqnarray}
 We estimate each term separately. First, in view of Lemma \ref{lem:main_est_cong}, we get
 \begin{equation}\label{equ:1st_est}
     \left|\sum_{|k| \leq K }  \left(1 -\frac{|k|}{M} \right)   (B_{\varepsilon,L}(k)-B(k))\right| \leq \sum_{|k| \leq K }    |B_{\varepsilon,L}(k)-B(k)|\ll K\left(L^{-(l-3)/2})+\sqrt{\varepsilon}\right)
 \end{equation}

Using the quantitative multiple mixing estimate Theorem \ref{thm:quan_mixing} (see Remark \ref{rem:quan_mix_cong}) and Lemma \ref{lem:f_epsilon_L_chi_hat_L_2} (see Remark \ref{rem:trun_sei_con}) in the congruence setup, we get
\begin{equation}\label{equ:B_epsilon_L}
    |B_{\varepsilon,L}(k)|\ll  \left|\int_{X_c} (\Psi \circ c_0^k) \Psi \, d\mu_c \right| \ll e^{-\delta_0|k|}S_q(\Psi)^2 \ll e^{-\delta_0|k|} \varepsilon^{-2q}L^2.
\end{equation}
 Using \eqref{equ:B_epsilon_L}, we get
 \begin{equation}\label{equ:2nd_est}
     \left|\sum_{K<|k|< M }  \left(1 -\frac{|k|}{M} \right)   B_{\varepsilon,L}(k) \right|\leq \sum_{|k|> K }    |B_{\varepsilon,L}(k)|\ll  \varepsilon^{-2q}L^2 e^{-\delta K}.
 \end{equation}
 
\noindent Since $\sum_{k=-\infty}^{\infty}  B(k) < \infty$ and $K=[b\log M] \rightarrow \infty$ as $M \rightarrow \infty,$
\begin{equation}\label{equ:3rd_est}
    \sum_{|k|\leq K} \frac{|k|}{M} B(k) \leq \frac{K}{M}\sum_{k=-\infty}^{\infty} B(k) =\vol(\Omega_2) \sigma_c^2 \frac{K}{M} \ll \frac{\log M}{M} \rightarrow 0 \quad \text{as} \quad M \rightarrow \infty
\end{equation}
 and
 \begin{equation}\label{equ:4th_est}
    \sum_{|k|>K}B(k) \rightarrow 0 \quad \text{as}\quad M\rightarrow\infty. 
 \end{equation} 
Combining \eqref{equ:1st_est}, \eqref{equ:2nd_est}, and \eqref{equ:3rd_est}, we get
\begin{eqnarray}
     \left|\|R_{M}^{(\varepsilon,L)}\|_2^2-\sum_{k=-\infty}^{\infty}B(k) \right|&\ll& K\left(L^{-(l-3)/2})+\sqrt{\varepsilon}\right) + \varepsilon^{-2q}L^2 e^{-\delta K} \nonumber \\ &+& \vol(\Omega_2) \sigma_c^2 \frac{K}{M} +  \sum_{|k|>K}B(k).
\end{eqnarray}
Now choose 
\[\varepsilon = M^{-2} \text{ and } L=M^{1/(l-2)} \log M.\]
Then 
\[K\sqrt{\varepsilon} \rightarrow 0 \quad \text{and}\quad  KL^{-(l-3)/2} \rightarrow 0  \]
as $M \rightarrow \infty$ for $l \geq 4.$

Also,
$\varepsilon^{-2q}L^2 e^{-\delta K} \ll M^{4q+\frac{2}{l-2}-b\delta} (\log M)^2.$
Choose \[b>\frac{4q+2/(l-2)}{\delta}.\]
Then $\varepsilon^{-2q}L^2 e^{-\delta K} \rightarrow 0$ as $M\rightarrow \infty.$

Finally in view \eqref{equ:3rd_est} and \eqref{equ:4th_est}, it
follows that 
\[\|R_M^{(\varepsilon,L)}\|_2^2 \rightarrow \vol(\Omega_2)\sigma_c^2=\sum_{k=-\infty}^{\infty}B(k) \quad \text{as } M\rightarrow \infty.\]
This completes the proof.


\begin{thebibliography}{99}

\bibitem{AG}
G. Aggarwal, A. Ghosh, \textit{Two central limit theorems in Diophantine approximations}, \emph{Illinois Journal of Mathematics} 69 (2025), no.3, 399-447 

\bibitem{AGH}
M. Alam, A. Ghosh, and J. Han. \textit{Higher moment formulae and limiting distributions
of lattice points},  J. Inst. Math. Jussieu 23 (2024), no. 5, 2081--2125.

\bibitem{A}
J. S. Athreya, \textit{Random affine lattices, Geometry}, groups and dynamics, 2015, pp. 169–174. MR3379825

\bibitem{BEG}
M. Bj$\ddot{\text{o}}$rklund, M. Einsiedler, A. Gorodnik, \textit{Quantitative multiple mixing}, J. Eur. Math. Soc., 22 (2020).

\bibitem{BG1}
M. Bj$\ddot{\text{o}}$rklund, A. Gorodnik, \textit{Central limit theorems for Diophantine approximants},  Math. Ann., 374 (2019).

\bibitem{BG1'}
M. Bj$\ddot{\text{o}}$rklund, A. Gorodnik, \textit{Central limit theorems for group actions which are exponentially mixing of all orders},  J. Anal. Math., 141 (2020), 457-482

\bibitem{BG2}
M. Bj$\ddot{\text{o}}$rklund, A. Gorodnik, \textit{Central limit theorems for generic lattice point counting},  Selecta Mathematica (2023).

\bibitem{DFV}
D. Dolgopyat, B. Fayad, and I. Vinogradov, \textit{Central limit theorems for simultaneous Diophantine approximations}, J. \'{E}c. polytech. Math. 4 (2017), 1–36. MR3583273


\bibitem{EMV}
D. El-Baz, J. Marklof, and I. Vinogradov, \textit{The distribution of directions in an affine lattice:
two-point correlations and mixed moments}, Int. Math. Res. Not. IMRN 5 (2015), 1371–1400. MR3340360

\bibitem{EMM}
A. Eskin, G. Margulis and S. Mozes,  \textit{Upper bounds and asymptotics in a quantitative version of the Oppenheim conjecture},  Ann. of Math., 147 (1998), 93-141.

 

\bibitem{FS}
M. Fr\'{e}chet, J. Shohat, \textit{A proof of the generalized second-limit theorem in the theory of probability}, Trans. Amer. Math. Soc., 33 (1931), 533-543.

\bibitem{GH}
A. Ghosh and J. Han, \textit{Values of inhomogeneous forms at S-integral points}, Mathematika 68
(2022), no. 2, 565–593. MR4418458

\bibitem{GKY}
A. Ghosh, D. Kelmer, and S. Yu, \textit{Effective density for inhomogeneous quadratic forms I: Generic forms and fixed shifts}, Int. Math. Res. Not. IMRN 6 (2022), 4682–4719. MR4391899

\bibitem{H}
K. Holm, \textit{Central limit theorem for counting functions related to symplectic lattices and bounded sets}, Discrete and Continuous Dynamical Systems
Vol. 43, No. 10, October 2023, pp. 3667-3705.
 
\bibitem{KW} Dmitry Kleinbock and Nick Wadleigh, \textit{An inhomogeneous Dirichlet theorem via shrinking targets},
Compos. Math., 155(7):1402–1423, 2019.

\bibitem{R1}
C. Rogers, \textit{Mean values over the space of lattices}, Acta Math. 94 (1955) 249-287.


\bibitem{S1}
 W. M. Schmidt, \textit{A metrical theorem in geometry of numbers}, Trans. Amer. Math. Soc., 95 (1960), 516-529

\bibitem{S2}
 W. M. Schmidt, \textit{Asymptotic formulae for point lattices of bounded determinant and subspaces of bounded height }, Duke Math J., 35 (1968), 327-339

\bibitem{Si}
C. Siegel, A mean value theorem in geometry of numbers, Ann. Math. 46 (1945) 340-347.
\end{thebibliography}
\end{document}